\documentclass[11pt]{amsart}
\usepackage[utf8]{inputenc}
\usepackage[usenames,dvipsnames,svgnames,table]{xcolor}
\usepackage[colorlinks=true, pdfstartview=FitV, linkcolor=blue, citecolor=blue, urlcolor=blue]{hyperref}

\usepackage{geometry}                
\usepackage{graphicx}
\usepackage{amssymb}
\usepackage{epstopdf}
\usepackage{lscape}
\usepackage[utf8]{inputenc}
\usepackage{tikz,caption}
\usepackage{enumitem}
\setlist[itemize]{leftmargin=2em}
\setlist[enumerate]{leftmargin=2em}
\usepackage{booktabs}
\usepackage{multirow}
\usepackage{mathtools}

\definecolor{darkblue}{rgb}{0.0,0,0.7} 
\definecolor{darkred}{rgb}{0.7,0,0} 
\definecolor{darkgreen}{rgb}{0, .6, 0} 

\newcommand{\defncolor}{\color{darkred}}
\newcommand{\defn}[1]{{\defncolor\emph{#1}}} 

\newtheorem{theorem}{Theorem}[section]
\newtheorem{prop}[theorem]{Proposition}
\newtheorem{cor}[theorem]{Corollary}
\newtheorem{lemma}[theorem]{Lemma}

\theoremstyle{definition}

\newtheorem{example}[theorem]{Example}
\newtheorem{remark}[theorem]{Remark}
\numberwithin{equation}{section}

\newcommand{\idiot}[1]{\vspace{5 mm}\par \noindent
\marginpar{\textsc{Note}}
\framebox{\begin{minipage}[c]{0.95 \textwidth}
#1 \end{minipage}}\vspace{5 mm}\par}

\renewcommand{\idiot}[1]{}

\def\la{{\lambda}}

\def\ga{{\gamma}}

\def\CC{{\mathbb C}}

\def\PP{{\mathrm{Par}}}
\def\AA{{\mathcal A}}

\def\bP{{\mathcal P}}

\def\dcl{{\{\!\!\{}}
\def\dcr{{\}\!\!\}}}

\def\o{\overline}
\def\b{\bar}

\def\P{{\rm{P}}}
\newcommand{\End}{\operatorname{End}}
\newcommand{\Hom}{\operatorname{Hom}}

\def\MP{{\rm M\!P\!}}

\def\bc{{\bf c}}

\usepackage[colorinlistoftodos]{todonotes}

\newdimen\squaresize \squaresize=10pt
\newdimen\thickness \thickness=0.4pt

\def\square#1{\hbox{\vrule width \thickness
     \vbox to \squaresize{\hrule height \thickness\vss
        \hbox to \squaresize{\hss#1\hss}
     \vss\hrule height\thickness}
\unskip\vrule width \thickness}
\kern-\thickness}

\def\vsquare#1{\vbox{\square{$#1$}}\kern-\thickness}

\def\thisbox#1{\kern-.09ex\fbox{#1}}
\def\downbox#1{\lower1.200em\hbox{#1}}
\newdimen\Squaresize \Squaresize=20pt
\newdimen\Thickness \Thickness=0.4pt

\def\Square#1{\hbox{\vrule width \Thickness
     \vbox to \Squaresize{\hrule height \Thickness\vss
        \hbox to \Squaresize{\hss#1\hss}
     \vss\hrule height\Thickness}
\unskip\vrule width \Thickness}
\kern-\Thickness}

\def\Vsquare#1{\vbox{\Square{$#1$}}\kern-\Thickness}

\title[A Multiset Partition Algebra]{Howe duality of the symmetric group and a multiset partition algebra}
\author[Rosa Orellana]{Rosa Orellana}%
\address{Dartmouth College, Mathematics Department, Hanover, NH 03755, USA} \email{rosa.c.orellana@dartmouth.edu}%

\author[Mike Zabrocki]{Mike Zabrocki}%
\address{Department of Mathematics and Statistics, York University, Toronto, Ontario M3J 1P3,
Canada} \email{zabrocki@mathstat.yorku.ca}%

\date{}  

\thanks{Work supported by NSF grants DMS-1300512 and DMS-1700058, and by NSERC}

\begin{document}
\begin{abstract}
We introduce the \emph{multiset partition algebra}, $\MP_{r,k}(x)$, that has bases elements indexed by multiset partitions, where $x$ is an indeterminate and $r$ and $k$ are non-negative integers.  This algebra can be realized as a diagram algebra that generalizes the partition algebra.  When $x$ is an integer greater or equal to $2r$, we show that $\MP_{r,k}(x)$ is isomorphic to a centralizer algebra of the symmetric group, $S_n$, acting on the polynomial ring on the variables $x_{ij}$, $1\leq i \leq n$ and $1\leq j\leq k$.  We describe the representations of $\MP_{r,k}(x)$, branching rule and restriction of its representations in the case that $x$ is an integer greater or equal to $2r$.

\end{abstract}

\maketitle
\tableofcontents

\section{Introduction}

Let $V_n$ be an $n$ dimensional vector space,
then Schur-Weyl duality is a fundamental property in representation theory
that relates the representations of the
general linear group $GL_n(\CC)$ and the symmetric group algebra
$\CC S_k$ as they both act on the tensor space
$$V_n^{\otimes k} = \underbrace{V_n \otimes V_n \otimes \cdots \otimes V_n}_{k\hbox{ times }}~.$$
The duality of these actions implies that
\begin{equation}\label{eq:decomp1}
V_n^{\otimes k} \cong \bigoplus_{\lambda} W_{GL_n(\CC)}^\lambda \otimes W_{\CC S_k}^\lambda
\end{equation}
where for a group or an algebra $A$, we use the notation $W^\lambda_A$ to represent an irreducible
representation of $A$ and the direct sum is over all partitions $\lambda$ of $k$.

This duality applies to algebras acting on spaces other than $V_n^{\otimes k}$
which centralize each other.
For instance, if we let $V_{n,k} = (\mathbb{C}^n)^k$ denote the $nk$ dimensional vector space of
sequences $(v_1, v_2, \ldots, v_k)$ where $v_i \in \mathbb{C}^n$.
There is an action of $GL_n(\CC)$ and of $GL_k(\CC)$ on $V_{n,k}$ which mutually commute.
Let $\bP(V)$ be the algebra of polynomial functions on $V$ and
$$\bP^r(V) = \{ f \in \bP(V) : f(zv) = z^r f(v)\hbox{ for }z \in \CC^{\times}\}$$
be the homogeneous polynomials of degree $r$.
The $GL_n(\CC) \times GL_k(\CC)$ action extends to $\bP^r(V_{n,k})$ \cite[Section 5.6.2]{GoodWall2}.
Then the analogous decomposition of $\bP^r(V_{n,k})$ into $GL_n(\CC) \times GL_k(\CC)$-modules is
\begin{equation}\label{eq:decomp2}
\bP^r(V_{n,k}) \cong \bigoplus_\lambda W_{GL_n(\CC)}^\lambda \otimes W_{GL_k(\CC)}^\lambda
\end{equation}
where the sum is over all partitions $\lambda$ of $r$ of length at most $min(n,k)$.
It is this duality that  is referred to as `Howe duality' \cite{Howe}.

The decompositions in the isomorphisms in Equations \eqref{eq:decomp1} and \eqref{eq:decomp2}
create a correspondence between representations of one algebra (or group) and the
representations of the dual algebra (or group).  In addition, the multiplicities of irreducible
representations of one algebra correspond to dimensions of irreducible representations
of the dual algebra.

In the 1990's, Martin \cite{Ma1} introduced the partition algebra, $\P_k(n)$,
as a generalization of the Temperley-Lieb algebra that could be used to study the transfer matrix
of certain statistical mechanics models.
In the case that $n\geq2k$ the partition algebra
is the algebra whose action centralizes the action of $S_n \subseteq GL_n(\CC)$ as the subgroup of permutation matrices
acting on $V_n^{\otimes k}$ \cite{Ma4, MR, Jones}.

In this case we also have a decomposition of $V^{\otimes k}$ into
$\CC S_n \times \P_k(n)$  irreducible representations,
\begin{equation}\label{eq:decomp3}
V_n^{\otimes k} \cong \bigoplus_{\lambda} W_{\CC S_n}^\lambda \otimes W_{\P_k(n)}^\lambda~.
\end{equation}
One application of this decomposition is that it helps to establish the relationship
between multiplicities occurring in the decomposition of the tensor product of $S_n$
irreducible representations (i.e. Kronecker coefficients) to the multiplicities occurring in the
restriction of $P_k(n)$ irreducibles  to Young subalgebras \cite{BDO}.
The multiplicities for tensor products of polynomial representations of $GL_n$ are known as
the Littlewood-Richardson coefficients and they are better understood than
the Kronecker coefficients.
One goal of the research in this area has been to develop
an understanding of algebras, combinatorics, and the representation theory
associated to $\P_k(n)$ to
help understand the representation theory of the
symmetric group in general
\cite{COSSZ,Eny, Ma2, Ma4, ME, Hal, HJ, HL, HR, BH1, BH2, BHH}.

In this paper we introduce a multiset partition algebra $\MP_{r,k}(x)$ such that,
for $n \geq 2r$, $\MP_{r,k}(n)$ is the centralizer algebra
of the symmetric group $S_n \subseteq GL_n(\CC)$ when it acts on $\bP^r(V_{n,k})$.
That is, if we denote the irreducible representations of this algebra by
$W_{\MP_{r,k}(n)}^\lambda$, then as a $\CC S_n \times \MP_{r,k}(n)$ module,
\begin{equation}\label{eq:decomp4}
\bP^r(V_{n,k}) \cong \bigoplus_\lambda W_{\CC S_n}^\lambda \otimes W_{\MP_{r,k}(n)}^\lambda~.
\end{equation}

We begin by introducing the generic multiset partition algebra $\MP_{r,k}(x)$ that depends
on a parameter $x$ and
uses multiset partitions to index basis elements.
We show that when $x = n$, an integer greater than or equal to $2r$,
$\MP_{r,k}(n)$ is isomorphic to $\AA_{r,k}(n):=\End_{S_n}(\bP^r(V_{n,k}))$,
the centralizer algebra
of the action of the symmetric group $S_n$ as the subgroup of permutation
matrices inside of $GL_n(\CC)$ when it acts on $\bP^r(V_{n,k})$.

Our hope is that this algebra can provide some insight into some aspects of the
representation theory of the symmetric group that are still not well understood.
Indeed, we show that the multiplicities when we restrict an irreducible multiset partition algebra representation
to sums of Young-type subalgebras are equal to the Kronecker coefficients.

A recent paper by Narayanan, Paul and Srivastava \cite{NPS, Paul} describes the centralizer algebra
$\End_{S_n}( \bP^{\alpha_1}(V_n) \otimes \bP^{\alpha_2}(V_n) \otimes \cdots \otimes \bP^{\alpha_k}(V_n))$
for a fixed weak composition $\alpha$ of length $k$.
This should be isomorphic to a subalgebra of the multiset partition algebra
introduced in this paper, but their product does not precisely agree with ours
and so the relationship is left open.

We did not consider in the definition of this algebra an extension corresponding to the
linear transformations in $\Hom( \bP^r( V_{n,k} ), \bP^s( V_{n,k}))$ where $r,s$
are non-negative integers which
commute with the action of $S_n$.
This centralizer algebra would depend on
$4$ parameters, $r,s,k,n$, and is spanned by multiset partitions with $r$ elements from
$\{1,2,\ldots,k\}$ and $s$ elements from $\{\o1,\o2,\ldots,\o{k}\}$.
This adds a bit to the complexity of the presentation, but we
did not have much more to say about it beyond a formula for the
dimension (see the comment in Remark \ref{rem:dim}).

The main results in this paper are:
\begin{enumerate}
\item A definition of the generic multiset partition algebra $\MP_{r,k}(x)$
and a proof that the product is associative (Proposition \ref{prop:associative}).
\item A definition of the centralizer algebra $\AA_{r,k}(n)$ and a proof
that $\MP_{r,k}(n) \cong \AA_{r,k}(n)$ when $n \geq 2r$ (Theorem \ref{th:isomorphism}).
\item A generating function formula for the dimension of $\AA_{r,k}(n)$
(Corollary \ref{cor:dimformula}).
\item A generating function formula for the dimension of an irreducible $\AA_{r,k}(n)$ indexed by a partition
$\lambda$ of $n$ (Proposition \ref{prop:dimirred}).
\item A formula for the branching coefficient of an irreducible
$\AA_{r,k}(n)$ to an $\AA_{r-d,k-1}(n)$-module (Theorem \ref{th:branchingrule}).
\item A proof that the multiplicities of an irreducible
$\AA_{r,k+\ell}(n)$-module restricted to a direct sum of $\AA_{d,k}(n) \otimes \AA_{r-d,\ell}(n)$-modules
are the Kronecker coefficients (Theorem \ref{th:kronrestrict}).
\end{enumerate}

In Section \ref{sec:notation} of this paper, we begin by introducing some of the combinatorial
notation of multisets, multiset partitions, diagrams and colored
multiset partitions and diagrams that we will need to encode
the basis elements and their product.  Section \ref{sec:msp_algebra}
presents the definitions of the generic form of the
multiset partition algebra.  Then in Section \ref{sec:centralizer} we show
that the generic multiset partition algebra is isomorphic to
the centralizer algebra $\AA_{r,k}(n)$ when $x=n \geq 2r$.
In Section \ref{sec:irreps} we give some enumerative results
and formulae for the dimensions of the irreducible representations
of $\AA_{r,k}(n)$.  In addition, we describe a branching rule and
the show that Kronecker coefficients occur when restricting irreducible
representations to direct sum of Young-type subalgebras.

\vskip .2in
\noindent
{\bf Acknowledgement:} The authors would like to thank Yohana Solomon whose helpful conversations
advanced some of the arguments in this paper.

\section{Preliminaries and Definitions}\label{sec:notation}

\subsection{Multisets and multiset partitions } \label{ssec:multset}
A \defn{multiset} is a collection of unordered elements and the elements can be repeated.   The collection of elements in a
multisets will be enclosed in $\dcl, \dcr$ to differentiate this structure from a set.
We say that $S$ is a multiset of $[k] := \{1,2,\ldots,k\}$
of size $r$ if $S$ contains $r$ elements (counting repetitions)
from the set $[k]$ (e.g. $\dcl 1,1,3,3,3\dcr$ is a multiset of $[3]$ with 5 elements).
The multiset $S$ will sometimes be represented using exponential notation $S = \dcl 1^{a_1}, 2^{a_2}, \ldots, k^{a_k} \dcr$
where the exponent $a_i$ indicates that the element $i$
occurs with multiplicity $a_i$ in $S$.  If the size of the multiset is $r$
then $a_1+a_2 + \cdots + a_k = r$.  The set of multisets
is endowed with the operation of taking union and if
$S=\dcl 1^{a_1}, 2^{a_2}, \ldots, k^{a_k} \dcr$ and $T = \dcl 1^{b_1}, 2^{b_2}, \ldots, k^{b_k} \dcr$
are multisets, then
$S \uplus T = \dcl 1^{a_1+b_1}, 2^{a_2+b_2}, \ldots, k^{a_k+b_k} \dcr$.
Multisets (like the monomials that they represent) are not uniquely totally ordered.
To assume a convention, given two multisets $S$ and $T$
 we say that $S< T$ if $S$ is empty and $T$ is not or $max(S)<max(T)$ or $max(S)=max(T)=m$ and
$S \backslash \dcl m \dcr < T \backslash \dcl m \dcr$.
We call this the \defn{last letter order}.

A \defn{multiset partition} of a multiset $S$ is a multiset of
multisets, $\dcl S^{(1)}, S^{(2)}, \ldots, S^{(d)}\dcr$, such that
$S^{(1)} \uplus S^{(2)} \uplus \cdots \uplus S^{(d)} = S$
where each of the $S^{(i)}$ are non-empty multisets.
We will assume the convention that if
$\pi = \dcl S^{(1)}, S^{(2)}, \ldots, S^{(d)}\dcr$ then
the parts of $\pi$ are listed in last letter order.
The nonempty multisets $S^{(i)}$ in a multiset partition are called
\defn{blocks} of the mulitset partition.    The \defn{length} of a
multiset partition is the number
of blocks (counted with multiplicity) and will be denoted by $\ell(\pi)$.
The \defn{content} of a multiset partition is the multiset
$S$ that it partitions.
For example, $\dcl \dcl 1\dcr, \dcl 1,2\dcr, \dcl 1,2\dcr, \dcl 3,3\dcr \dcr$
is a multiset partition with content $\dcl 1^3, 2^2, 3^2\dcr$ and
length 4.

For a multiset $S$ which is a block of a multiset partition $\pi$, let $m_S(\pi)$ represent
the number of times that $S$ occurs in $\pi$.  If $S^{(1)}, S^{(2)}, \ldots, S^{(r)}$
are the distinct multisets that appear in the multiset partition $\pi$, define
$$m(\pi)! = m_{S^{(1)}}(\pi)! m_{S^{(2)}}(\pi)! \cdots m_{S^{(r)}}(\pi)!~.$$

In this article we are interested in multiset partitions of multisets of  the set
$[k]\cup[\o k]:= \{1, 2, \ldots, k\} \cup \{\o1, \o2,\ldots, \o k\}$  such that $r$ elements
come from $[k]$ and $r$ elements come from $[\o k]$.  Thus the content
of the multiset partition is a multiset of size $2r$.   A multiset
partition of this type is said to have \defn{size} $r$ and
\defn{maximum value} $k$.   We will denote the set of
multiset partitions of size $r$ and max
value $k$ by $\Pi_{r,k}$.  The notation $\Pi_{r,k,n}$
will represent the subset of $\Pi_{r,k}$ of
multiset partitions with maximum length $n$.

\subsection{Diagrams}
The notion of a diagram that we use here comes from diagram algebras
(e.g. the Brauer algebra or partition algebra).
This is a graph theoretic representation of multiset partitions in the
same sense that diagrams have classically been used to represent
basis elements of centralizer algebras.

A multiset partition in $\Pi_{r,k,n}$ can be represented by a graph
with two horizontal rows with $r$ vertices each. The vertices
are labeled with the elements in the parts of the multisets in such a
way that the top row has only unbarred labels and the
bottom row has only barred labels and the numbers on both rows are
weakly increasing from left to right according to the order
$1<2<\ldots < k <\o1<\o2<\ldots< \o k $.   Two vertices in this graph
are connected by a path  if they are in the same block of
the same multiset partition. Thus, it follows that each block of the
multiset partition defines a connected component of the graph.
We remark that there are many graphs that can represent the same multiset partition, in fact any two graphs that
are have the same labeled connected components will represent the same multiset partition.
We define the \defn{diagram} of a multiset partition as the
equivalence class of graphs that satisfy the above conditions.
In our examples the blocks of the multiset partition will be connected by
paths or cycles.  However, as it is usual in diagram algebras,
we only care that there is a connected component containing the vertices for each
block of the multiset partition and not how they are connected.

\begin{example}\label{ex:firstmsp}
Consider the multiset partition
$\pi = \dcl \dcl 1\dcr, \dcl 1, 1\dcr, \dcl 2, \o1, \o1 \dcr, \dcl 4, \o2, \o4\dcr, \dcl \o4\dcr \dcr$ of $[4]\cup[\o 4]$
which is an element of $\Pi_{5, 4}$ of length $5$, i.e., an element in $\Pi_{5,4,n}$ with $n\geq5$.  This can also be represented by
any of the following graphs:
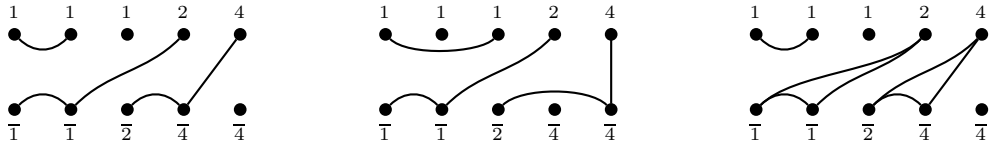
\begin{center}
\begin{tikzpicture}[scale = 0.5,thick, baseline={(0,-1ex/2)}]
\tikzstyle{vertex} = [shape = circle, minimum size = 4pt, inner sep = 1pt]
\node[vertex] (G--5) at (6.0, -1) [shape = circle, draw, fill] {};
\node[left] at (6.4, -1.6) {\tiny $\overline 4$};
\node[vertex] (G--4) at (4.5, -1) [shape = circle, draw, fill] {};
\node[left] at (4.9, -1.6) {\tiny $\overline 4$};
\node[vertex] (G--3) at (3.0, -1) [shape = circle, draw, fill] {};
\node[left] at (3.4, -1.6)  {\tiny $\overline 2$};
\node[vertex] (G-5) at (6.0, 1) [shape = circle, draw, fill] {};
\node[left] at (6.4, 1.6) {\tiny 4};
\node[vertex] (G-4) at (4.5, 1) [shape = circle, draw, fill] {};
\node[left] at (4.9, 1.6) {\tiny 2};
\node[vertex] (G--2) at (1.5, -1) [shape = circle, draw, fill] {};
\node[left] at (1.9, -1.6)  {\tiny $\overline 1$};
\node[vertex] (G--1) at (0.0, -1) [shape = circle, draw, fill] {};
\node[left] at (0.4, -1.6)  {\tiny $\overline 1$};
\node[vertex] (G-1) at (0.0, 1) [shape = circle, draw, fill] {};
\node[left] at (0.4, 1.6) {\tiny 1};
\node[vertex] (G-2) at (1.5, 1) [shape = circle, draw, fill] {};
\node[left] at (1.9, 1.6) {\tiny 1};
\node[vertex] (G-3) at (3.0, 1) [shape = circle, draw, fill] {};
\node[left] at (3.4, 1.6) {\tiny 1};
\draw (G-5) .. controls +(-0.75, -1) and +(0.75, 1) .. (G--4);
\draw (G--4) .. controls +(-0.5, 0.5) and +(0.5, 0.5) .. (G--3);
\draw (G-4) .. controls +(-1, -1) and +(1, 1) .. (G--2);
\draw (G--2) .. controls +(-0.5, 0.5) and +(0.5, 0.5) .. (G--1);
\draw (G-1) .. controls +(0.5, -0.5) and +(-0.5, -0.5) .. (G-2);
\end{tikzpicture}\hspace{40pt}
\begin{tikzpicture}[scale = 0.5,thick, baseline={(0,-1ex/2)}]
\tikzstyle{vertex} = [shape = circle, minimum size = 4pt, inner sep = 1pt]
\node[vertex] (G--5) at (6.0, -1) [shape = circle, draw, fill] {};
\node[left] at (6.4, -1.6) {\tiny $\overline 4$};
\node[vertex] (G--4) at (4.5, -1) [shape = circle, draw, fill] {};
\node[left] at (4.9, -1.6) {\tiny $\overline 4$};
\node[vertex] (G--3) at (3.0, -1) [shape = circle, draw, fill] {};
\node[left] at (3.4, -1.6)  {\tiny $\overline 2$};
\node[vertex] (G-5) at (6.0, 1) [shape = circle, draw, fill] {};
\node[left] at (6.4, 1.6) {\tiny 4};
\node[vertex] (G-4) at (4.5, 1) [shape = circle, draw, fill] {};
\node[left] at (4.9, 1.6) {\tiny 2};
\node[vertex] (G--2) at (1.5, -1) [shape = circle, draw, fill] {};
\node[left] at (1.9, -1.6)  {\tiny $\overline 1$};
\node[vertex] (G--1) at (0.0, -1) [shape = circle, draw, fill] {};
\node[left] at (0.4, -1.6)  {\tiny $\overline 1$};
\node[vertex] (G-1) at (0.0, 1) [shape = circle, draw, fill] {};
\node[left] at (0.4, 1.6) {\tiny 1};
\node[vertex] (G-2) at (1.5, 1) [shape = circle, draw, fill] {};
\node[left] at (1.9, 1.6) {\tiny 1};
\node[vertex] (G-3) at (3.0, 1) [shape = circle, draw, fill] {};
\node[left] at (3.4, 1.6) {\tiny 1};
\draw (G-5) .. controls +(0, 0) and +(0 ,0) .. (G--5);
\draw (G--5) .. controls +(-0.6, 0.6) and +(0.6, 0.6) .. (G--3);
\draw (G-4) .. controls +(-1, -1) and +(1, 1) .. (G--2);
\draw (G--2) .. controls +(-0.5, 0.5) and +(0.5, 0.5) .. (G--1);
\draw (G-1) .. controls +(0.6, -0.6) and +(-0.5, -0.5) .. (G-3);
\end{tikzpicture}\hspace{40pt}
\begin{tikzpicture}[scale = 0.5,thick, baseline={(0,-1ex/2)}]
\tikzstyle{vertex} = [shape = circle, minimum size = 4pt, inner sep = 1pt]
\node[vertex] (G--5) at (6.0, -1) [shape = circle, draw, fill] {};
\node[left] at (6.4, -1.6) {\tiny $\overline 4$};
\node[vertex] (G--4) at (4.5, -1) [shape = circle, draw, fill] {};
\node[left] at (4.9, -1.6) {\tiny $\overline 4$};
\node[vertex] (G--3) at (3.0, -1) [shape = circle, draw, fill] {};
\node[left] at (3.4, -1.6)  {\tiny $\overline 2$};
\node[vertex] (G-5) at (6.0, 1) [shape = circle, draw, fill] {};
\node[left] at (6.4, 1.6) {\tiny 4};
\node[vertex] (G-4) at (4.5, 1) [shape = circle, draw, fill] {};
\node[left] at (4.9, 1.6) {\tiny 2};
\node[vertex] (G--2) at (1.5, -1) [shape = circle, draw, fill] {};
\node[left] at (1.9, -1.6)  {\tiny $\overline 1$};
\node[vertex] (G--1) at (0.0, -1) [shape = circle, draw, fill] {};
\node[left] at (0.4, -1.6)  {\tiny $\overline 1$};
\node[vertex] (G-1) at (0.0, 1) [shape = circle, draw, fill] {};
\node[left] at (0.4, 1.6) {\tiny 1};
\node[vertex] (G-2) at (1.5, 1) [shape = circle, draw, fill] {};
\node[left] at (1.9, 1.6) {\tiny 1};
\node[vertex] (G-3) at (3.0, 1) [shape = circle, draw, fill] {};
\node[left] at (3.4, 1.6) {\tiny 1};
\draw (G-5) .. controls +(-0.75, -1) and +(0.75, 1) .. (G--4);
\draw (G--4) .. controls +(-0.5, 0.5) and +(0.5, 0.5) .. (G--3);
\draw (G--3) .. controls +(1, 1) and +(-1, -1) .. (G-5);
\draw (G-4) .. controls +(-1, -1) and +(1, 1) .. (G--2);
\draw (G--2) .. controls +(-0.5, 0.5) and +(0.5, 0.5) .. (G--1);
\draw (G--1) .. controls +(1, 1) and +(-1, -1) .. (G-4);
\draw (G-1) .. controls +(0.5, -0.5) and +(-0.5, -0.5) .. (G-2);
\end{tikzpicture}
\end{center}
These three graphs all represent the same multiset partition.
The first two are isomorphic as labeled graphs, but the last one is not isomorphic as a graph.
All three are equivalent as multiset partitions as we only care about using the connectivity
to represent the blocks of the multiset partition. We note that there are more labeled graphs
that can be used to represent this multiset partition.
\end{example}

\subsection{Colored multiset partitions} \label{ssec:coloredmsp}
In order to define the product in our algebra we will need colored
multiset partitions.  Let $n$ be a non-negative integer.
A \defn{colored multiset partition} is an element in $\pi$ in $\Pi_{r,k,n}$
together with an assignment of a label from $[n]$ to each block of $\pi$, such that each
block has a different label (these labels are the `colors'). More precisely, given a multiset
partition  $\pi=\dcl S^{(1)}, \ldots, S^{(d)}\dcr \in \Pi_{r,k,n}$ with
$S^{(i)} \leq S^{(i+1)}$ and $\mathbf{c}=(c_1,\ldots, c_d)\in [n]^d$ and a
sequence such that $c_i\neq c_j$ if $i \neq j$ and $c_i<c_{i+1}$ if $S^{(i)}=S^{(i+1)}$, then the notation
$\pi^{\mathbf{c}} = \{S^{(1)}_{c_1}, \ldots, S^{(d)}_{c_d}\}$ will
denote that the color $c_i$ was assigned to block $S^{(i)}$. We remark
that a colored multiset partition is a set since the colors are
distinct and so the colored multisets are also distinct.
Notice that $n\geq d$, otherwise it is not possible to assign distinct colors to all the parts.
We will say that $\pi^{\mathbf{c}}$ is a coloring of $\pi$.

\begin{example}\label{ex:firstcmsp}
Let $n=7$, $\pi = \dcl \dcl 1\dcr, \dcl 1, 1\dcr, \dcl 2, \o1, \o1 \dcr, \dcl \o4\dcr, \dcl 4, \o2, \o4\dcr \dcr$
and $\mathbf{c}=(1,3,5,2,6)$, $\pi^{(1,3,5,2,6)}
=\{ \dcl 1\dcr_{\bf 1}, \dcl 1, 1\dcr_{\bf 3}, \dcl 2, \o1, \o1 \dcr_{\bf 5},
\dcl \o4\dcr_{\bf 2}, \dcl 4, \o2, \o4\dcr_{\bf 6} \}$ is a colored multiset partition.
This can also be represented using diagrams using bold integers to indicate the colors of the connected component.

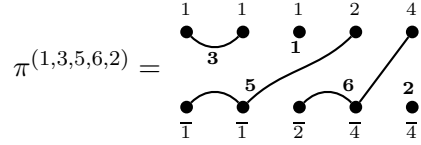
\begin{center}$\pi^{(1,3,5,6,2)} = $
\begin{tikzpicture}[scale = 0.5,thick, baseline={(0,-1ex/2)}]
\tikzstyle{vertex} = [shape = circle, minimum size = 4pt, inner sep = 1pt]
\node[vertex] (G--5) at (6.0, -1) [shape = circle, draw, fill] {};
\node[left] at (6.4, -1.6) {\tiny $\overline 4$};
\node[vertex] (G--4) at (4.5, -1) [shape = circle, draw, fill] {};
\node[left] at (4.9, -1.6) {\tiny $\overline 4$};
\node[vertex] (G--3) at (3.0, -1) [shape = circle, draw, fill] {};
\node[left] at (3.4, -1.6)  {\tiny $\overline 2$};
\node[vertex] (G-5) at (6.0, 1) [shape = circle, draw, fill] {};
\node[left] at (6.4, 1.6) {\tiny 4};
\node[vertex] (G-4) at (4.5, 1) [shape = circle, draw, fill] {};
\node[left] at (4.9, 1.6) {\tiny 2};
\node[vertex] (G--2) at (1.5, -1) [shape = circle, draw, fill] {};
\node[left] at (1.9, -1.6)  {\tiny $\overline 1$};
\node[vertex] (G--1) at (0.0, -1) [shape = circle, draw, fill] {};
\node[left] at (0.4, -1.6)  {\tiny $\overline 1$};
\node[vertex] (G-1) at (0.0, 1) [shape = circle, draw, fill] {};
\node[left] at (0.4, 1.6) {\tiny 1};
\node[vertex] (G-2) at (1.5, 1) [shape = circle, draw, fill] {};
\node[left] at (1.9, 1.6) {\tiny 1};
\node[vertex] (G-3) at (3.0, 1) [shape = circle, draw, fill] {};
\node[left] at (3.4, 1.6) {\tiny 1};
\draw (G-5) .. controls +(-0.75, -1) and +(0.75, 1) .. (G--4);
\node at (0.7,0.3) {\tiny \bf 3};
\node at (1.7,-0.4) {\tiny \bf 5};
\node at (2.9,0.6) {\tiny \bf 1};
\node at (4.3,-0.4) {\tiny \bf 6};
\node at (5.9,-0.5) {\tiny \bf 2};
\draw (G--4) .. controls +(-0.5, 0.5) and +(0.5, 0.5) .. (G--3);
\draw (G-4) .. controls +(-1, -1) and +(1, 1) .. (G--2);
\draw (G--2) .. controls +(-0.5, 0.5) and +(0.5, 0.5) .. (G--1);
\draw (G-1) .. controls +(0.5, -0.5) and +(-0.5, -0.5) .. (G-2);
\end{tikzpicture}
\end{center}
\end{example}

\begin{remark}
If we order the blocks using the last letter order and since $\mathbf{c}$
is an ordered sequence, we get a unique colored multiset partition for
each multiset partition in $\Pi_{r,k,n}$ and $\mathbf{c}$.
However, when we draw the diagram of such a colored multiset partition
we could list the equal blocks colored with different colors in any order.
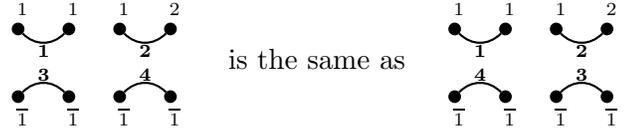
\begin{align*}
\begin{tikzpicture}[scale = 0.45,thick, baseline={(0,-1ex/2)}]
\tikzstyle{vertex} = [shape = circle, minimum size = 4pt, inner sep = 1pt]
\node[vertex] (G--4) at (4.5, -1) [shape = circle, draw,fill] {};
\node[vertex] (G--3) at (3.0, -1) [shape = circle, draw,fill] {};
\node[vertex] (G-3) at (3.0, 1) [shape = circle, draw,fill] {};
\node[vertex] (G-4) at (4.5, 1) [shape = circle, draw, fill] {};
\node[vertex] (G--2) at (1.5, -1) [shape = circle, draw,fill] {};
\node[vertex] (G--1) at (0.0, -1) [shape = circle, draw,fill] {};
\node[vertex] (G-1) at (0.0, 1) [shape = circle, draw, fill] {};
\node[vertex] (G-2) at (1.5, 1) [shape = circle, draw,fill] {};
\node[left] at (5.1, 1.6) {\tiny $2$};
\node[left] at (3.6, 1.6) {\tiny $1$};
\node[left] at (2.1, 1.6) {\tiny $1$};
\node[left] at (0.6, 1.6) {\tiny $1$};
\node[left] at (5.1, -1.6) {\tiny $\o1$};
\node[left] at (3.6, -1.6) {\tiny $\o1$};
\node[left] at (2.1, -1.6) {\tiny $\o1$};
\node[left] at (0.6, -1.6) {\tiny $\o1$};
\draw (G-3) .. controls +(0.5, -0.5) and +(-0.5, -0.5) .. (G-4);
\draw (G--4) .. controls +(-0.5, 0.5) and +(0.5, 0.5) .. (G--3);
\draw (G--2) .. controls +(-0.5, 0.5) and +(0.5, 0.5) .. (G--1);
\draw (G-1) .. controls +(0.5, -0.5) and +(-0.5, -0.5) .. (G-2);
\node at (0.75,0.35) {\tiny \bf 1};
\node at (3.75,0.35) {\tiny \bf 2};
\node at (0.75,-.35) {\tiny \bf 3};
\node at (3.75,-.35) {\tiny \bf 4};
\end{tikzpicture}
\quad \text{ is the same as }\quad
\begin{tikzpicture}[scale = 0.45,thick, baseline={(0,-1ex/2)}]
\tikzstyle{vertex} = [shape = circle, minimum size = 4pt, inner sep = 1pt]
\node[vertex] (G--4) at (4.5, -1) [shape = circle, draw,fill] {};
\node[vertex] (G--3) at (3.0, -1) [shape = circle, draw,fill] {};
\node[vertex] (G-3) at (3.0, 1) [shape = circle, draw,fill] {};
\node[vertex] (G-4) at (4.5, 1) [shape = circle, draw, fill] {};
\node[vertex] (G--2) at (1.5, -1) [shape = circle, draw,fill] {};
\node[vertex] (G--1) at (0.0, -1) [shape = circle, draw,fill] {};
\node[vertex] (G-1) at (0.0, 1) [shape = circle, draw, fill] {};
\node[vertex] (G-2) at (1.5, 1) [shape = circle, draw,fill] {};
\node[left] at (5.1, 1.6) {\tiny $2$};
\node[left] at (3.6, 1.6) {\tiny $1$};
\node[left] at (2.1, 1.6) {\tiny $1$};
\node[left] at (0.6, 1.6) {\tiny $1$};
\node[left] at (5.1, -1.6) {\tiny $\o1$};
\node[left] at (3.6, -1.6) {\tiny $\o1$};
\node[left] at (2.1, -1.6) {\tiny $\o1$};
\node[left] at (0.6, -1.6) {\tiny $\o1$};
\draw (G-3) .. controls +(0.5, -0.5) and +(-0.5, -0.5) .. (G-4);
\draw (G--4) .. controls +(-0.5, 0.5) and +(0.5, 0.5) .. (G--3);
\draw (G--2) .. controls +(-0.5, 0.5) and +(0.5, 0.5) .. (G--1);
\draw (G-1) .. controls +(0.5, -0.5) and +(-0.5, -0.5) .. (G-2);
\node at (0.75,0.35) {\tiny \bf 1};
\node at (3.75,0.35) {\tiny \bf 2};
\node at (0.75,-.35) {\tiny \bf 4};
\node at (3.75,-.35) {\tiny \bf 3};
\end{tikzpicture}~.
\end{align*}
\end{remark}

 Let $\pi\in \Pi_{r,k,n}$ and $\pi^\bc$ a coloring of $\pi$,  denote by $\pi^\bc_{top}$ to be the colored multiset
partition restricted to the entries in $[k]$
while $\pi^\bc_{bot}$ will be the colored multiset
partition restricted to the entries in $[\o{k}]$ and then ignoring the bars.
Similarly, we will use $\pi_{top}$ and $\pi_{bot}$ to represent the corresponding restrictions
for uncolored multiset partitions. Consider $\pi^{(1,3,5,6,2)}$ from Example \ref{ex:firstcmsp}, then
$\pi^\bc_{top} = \dcl  \dcl  1  \dcr_{\bf 1}, \dcl   1, 1 \dcr_{\bf 3} , \dcl  2  \dcr_{\bf 5} , \dcl   4 \dcr_{\bf 6}\dcr$
and $\pi^\bc_{bot} = \dcl  \dcl   1, 1  \dcr_{\bf 5}, \dcl  4  \dcr_{\bf 2}, \dcl  2, 4 \dcr_{\bf 6}  \dcr$,
notice that we have left out the bars in $\pi^\bc_{bot}$.

We use the notation $\pi^\bc \rightarrow \pi$ if $\pi$ is obtained from
$\pi^\bc$  by ignoring the colors.
That is,  if $\pi^\bc = \{ S^{(1)}_{c_1}, S^{(2)}_{c_2}, \ldots, S^{(d)}_{c_d} \}$ for some multisets
$S^{(i)}$ and values $c_i \in [n]$ then  $\pi = \dcl S^{(1)}, S^{(2)}, \ldots, S^{(d)} \dcr$.
In this case we also say that $\pi$ is the underlying multiset partition
of $\pi^\bc$.

\subsection{Symmetric Functions}
\label{subsec:sfnotation}
A \defn{partition} is a sequence $\lambda_1 \geq \la_2 \geq \cdots \geq \la_{\ell(\la)}$
where $\ell(\la)$ denotes the number of nonzero terms (called parts) of $\lambda$.  To indicate that
$\lambda$ is a partition we use the notation $\lambda \in \PP$.   If in addition, we have that
$\lambda_1 + \lambda_2 + \cdots + \lambda_{\ell(\lambda)} = n$, we say that $\lambda$ is a partition of $n$
and we use $\lambda \in \PP_n$ to indicate that $\lambda$ is a partition of $n$.  In general,
we use greek letters such as $\lambda$ and $\mu$ to denote partitions.

To prove the results about the representation theory or our algebra we
will use results from the theory of symmetric
functions.  For this we need to introduce notation and well known
results in the theory, these results can be found in  \cite{Mac}, \cite{Sta}.
We define the ring of symmetric functions as the polynomial
ring
$$\Lambda = {\mathbb Q}[ p_1, p_2, p_3, \ldots ]$$
where for each $i\geq 1$, the $p_i$ is the power sum generator of degree $i$.
For any partition $\lambda$, $p_\lambda = p_{\lambda_1} \cdots p_{\lambda_{\ell(\lambda)}}$.
The ring $\Lambda$ has  distinguished bases
$\{ p_\lambda \}_{\lambda \in \PP}$,
$\{ e_\lambda \}_{\lambda \in \PP}$,
$\{ h_\lambda \}_{\lambda \in \PP}$,
$\{ m_\lambda \}_{\lambda \in \PP}$ and
$\{ s_\lambda \}_{\lambda \in \PP}$
as the power sum, elementary, complete homogeneous, monomial and Schur bases respectively.
Let $z_\lambda = \prod_{i\geq1} i^{m_i(\lambda)} m_i(\lambda)!$
where $m_i(\lambda)$ is the number of parts of size $i$ in the partition $\lambda$.
We use the Hall-scalar product on symmetric functions
which is defined so that
$\left< s_\lambda, s_\mu \right> = \left< p_\lambda, p_\mu/z_\mu \right>
= \left< h_\lambda, m_\mu \right> = \delta_{\lambda\mu}$, where $\delta_{\lambda\mu}$ is 1 if
$\lambda = \mu$ and zero otherwise.

In addition to the ring structure of $\Lambda$ as a polynomial ring, there is a second
product called the \defn{Kronecker product} defined on the power sum basis by
$\frac{p_\mu}{z_\mu} \ast \frac{p_\lambda}{z_\lambda} = \delta_{\lambda\mu} \frac{p_\lambda}{z_\lambda}$.
It is defined so that for all $\lambda, \mu$ and $\nu$ which are partitions of
a positive integer $n$,
$$\left< s_\lambda \ast s_\mu, s_\nu \right> = \dim~\Hom( W^\nu_{\CC S_n}, W^\lambda_{\CC S_n} \otimes
W^\mu_{\CC S_n} ) := g_{\la\mu\nu}~.$$
where $W^\lambda_{\CC S_n}$ denotes the irreducible representation of the symmetric group algebra, $\CC S_n$,  indexed by
the partition $\lambda$.  The coefficients $g_{\la\mu\nu}$ are known as \defn{Kronecker coefficients}.

We will make use of `plethystic notation' for some of the
symmetric function expressions (see \cite{LR} for an expository reference).
Let $E(x_1, x_2, \ldots; q) = E(X;q)$ be an algebraic expression in indeterminates $q, x_1, x_2, \ldots$,
then for an element $f \in \Lambda$,
the evaluation of a symmetric function at an expression $E(X;q)$
is denoted $f[E(X;q)]$ and is defined
by expanding $f$ in the power sum basis and replacing each $p_r$ by
$E(x_1^r, x_2^r, \ldots; q^r)$.  That is, if $f = \sum_\lambda c_\lambda p_\lambda$, then
$$f[E(X;q)] = \sum_{\lambda} c_\lambda \prod_{i=1}^{\ell(\lambda)}
E(x_1^{\lambda_i}, x_2^{\lambda_i}, \ldots; q^{\lambda_i})~.$$
We will use $X$ (or $Y$) to represent the sum $x_1+ x_2+ x_3+\cdots$
(or $y_1+ y_2+ y_3+\cdots$) and use freely that ${\mathbb Q}[ p_1, p_2, p_3, \ldots ] \cong
{\mathbb Q}[ p_1[X], p_2[X], p_3[X], \ldots ]$.

The following symmetric function identities will be used to obtain a generating function expression in
Section \ref{sec:irreps}.  In the Lemma, for any integer $r$,  the $s_r$ denotes the Schur function
indexed by partition consisting of one part $r$.

\begin{lemma}\label{lem:sfids} \cite[Ch. 1, Sec. 4 and 5]{Mac} Let $\Omega = 1+s_1+s_2+s_3+\cdots$,
\begin{align*}
(a)& \hskip .2in \Omega[XY] = \prod_{i \geq 1}\prod_{j \geq 1} \frac{1}{1-x_iy_j}
= \sum_{\lambda\in\PP} s_\lambda[X] s_\lambda[Y],\\
(b)& \hskip .2in \left< \Omega[XY], f[X] \right> = f[Y],\\
(c)& \hskip .2in \Omega[X+Y] = \Omega[X]\Omega[Y]\hbox{ or }s_r[X+Y] = \sum_{a=0}^r s_a[X] s_{r-a}[Y]~.\\
\end{align*}
\end{lemma}

\section{The generic multiset partition algebra} \label{sec:msp_algebra}

Let $r$ and $k$ be non-negative integers and $x$ be a nonzero indeterminate.
In Section \ref{ssec:multset} we defined $\Pi_{r,k}$ as the set of
multiset partitions of size $r$ and maximum value $k$.
Define $\MP_{r,k}(x)$ to be the $\mathbb{C}(x)$-span of elements
indexed by  multiset partitions with coefficients in $\mathbb{C}(x)$, i.e,
${\rm span}_{\CC(x)}\{ X_\pi \hbox{ for }\pi \in \Pi_{r,k} \}$.
We set $\MP_{0,0}(x)=\mathbb{C}(x)$ and we identify it with the linear span of the empty multiset partition.

In what follows we will define a product on basis elements, $X_\pi$, where we identify
the diagram for $\pi$ with the basis element.  This product will
make $\MP_{r,k}(x)$ an associative algebra over $\mathbb{C}(x)$ with an identity element.

For $\gamma \in \Pi_{r,k}$, let $\o\gamma$ be the multiset partition of $[\o{k}] \cup [\b{\b{k}}]$
formed by adding an extra bar to all the entries of $\gamma$.
Let $A$ be a set and $S$ a multiset, then let $S|_A$ be the restriction of the multiset
$S$ to the elements which are in $A$ (e.g. $S|_A = \dcl i \in S : i \in A\dcr$)
and for a multiset partition $\gamma$, $\gamma|_A$ is the (possibly empty) multiset partition
$\dcl S|_A : S \in \gamma \dcr$, that is $\gamma|_A$ is the multiset partition $\gamma$
restricted to the elements in $A$.

The expression for the product will involve multiset partitions with entries in $[k] \cup [\o{k}] \cup [\b{\b{k}}]$.
Let $\Gamma_{r,k}$ denote the set of multiset partitions with $r$ entries from $[k]$,  $r$ entries from $[\o{k}]$
and $r$ entries from
$[\b{\b{k}}]$ and let $\nu$ be an element of $\Gamma_{r,k}$.
For each distinct multiset $S$ in $\nu |_{[k] \cup [\b{\b{k}}]}$, let
$\nu_S$ be the multiset partition consisting of the parts $T$ of $\nu$
such that $T |_{[k] \cup [\b{\b{k}}]} = S$, that is,
$$\nu_S = \dcl T \in \nu :  T |_{[k] \cup [\b{\b{k}}]} = S \dcr~.$$
Also let $\beta = \dcl S \in \nu : \forall i \in S, i \in [\o{k}] \dcr$ be the (possibly empty)
multiset partition with just the parts of $\nu$ which have elements only in $[\o{k}]$.
Then define the coefficient
\begin{equation}\label{eq:defanu}
a_\nu = \prod_{S} {\ell(\nu_S)!}/{m(\nu_S)!}
\end{equation}
where the product is over the distinct multisets $S$ which appear in $\nu |_{[k] \cup [\b{\b{k}}]}$.
Also define
\begin{equation}\label{eq:defbnu}
b_\nu(x) = (x-\ell(\nu |_{[k] \cup [\b{\b{k}}]}))_{\ell(\beta)}/m(\beta)!
\end{equation}
 where $(a)_r = a(a-1)\cdots(a-r+1)$.

\begin{example}  Let $r=9$ and $k=2$
and consider, for example, the multiset partition
$\nu = \dcl \dcl1 \dcr$,  $\dcl1,\o1 \dcr$,   $\dcl1,\o2 \dcr$, $\dcl1,\o2 \dcr$,
$\dcl1,\b{\b2} \dcr$,
$\dcl1,2,\b{\b1} \dcr$, $\dcl1,2, \b{\b1},\b{\b1}\dcr$,
$\dcl\o1 \dcr$, $\dcl\o1 \dcr$, $\dcl \o1, \o2 \dcr$,
$\dcl \o1, \o2, \b{\b1} \dcr$, $\dcl\b{\b2},\b{\b2} \dcr$,
$\dcl\b{\b2},\b{\b2} \dcr \dcr$.  Now
$\nu |_{[k] \cup [\b{\b{k}}]} = \dcl \dcl1 \dcr$,  $\dcl1 \dcr$,   $\dcl1 \dcr$, $\dcl1 \dcr$,
$\dcl1,\b{\b2} \dcr$,
$\dcl1,2,\b{\b1} \dcr$, $\dcl1,2, \b{\b1},\b{\b1}\dcr$,
$\dcl\b{\b1} \dcr$, $\dcl\b{\b2},\b{\b2} \dcr$,
$\dcl\b{\b2},\b{\b2} \dcr \dcr$ has $6$ distinct multisets and has length $10$.
Below we list the distinct multisets $S$ of $\nu |_{[k] \cup [\b{\b{k}}]}$ and the corresponding $\nu_S$.

\begin{center}
\begin{tabular}{c|c}
$S$&$\nu_S$\\
\hline
$\dcl1 \dcr$& $\dcl \dcl1 \dcr$,  $\dcl1,\o1 \dcr$,   $\dcl1,\o2 \dcr$, $\dcl1,\o2 \dcr \dcr$\\
$\dcl1,\b{\b2} \dcr$& $\dcl \dcl1,\b{\b2} \dcr \dcr$\\
$\dcl1,2,\b{\b1} \dcr$& $\dcl \dcl1,2,\b{\b1} \dcr \dcr$\\
$\dcl1,2, \b{\b1},\b{\b1}\dcr$& $\dcl \dcl1,2, \b{\b1},\b{\b1}\dcr \dcr$\\
$\dcl \b{\b1} \dcr$ & $\dcl \dcl \o1, \o2, \b{\b1} \dcr \dcr$\\
$\dcl\b{\b2},\b{\b2} \dcr$ & $\dcl \dcl\b{\b2},\b{\b2} \dcr$, $\dcl\b{\b2},\b{\b2} \dcr \dcr$
\end{tabular}
\end{center}

From this table we conclude that $a_\nu = 12$ (for the multiset $S = \dcl1 \dcr$ is ${\ell(\nu_S)!}/{m(\nu_S)!}=12$, while from
each of the other multisets $S$, ${\ell(\nu_S)!}/{m(\nu_S)!}=1$).

The multiset partition $\beta$ consisting of the parts with only entries in $[\o{2}]$
is $\beta = \dcl \dcl\o1 \dcr$, $\dcl\o1 \dcr$, $\dcl \o1, \o2 \dcr \dcr$, and hence
$b_\nu(x) = (x-10)(x-11)(x-12)/2$
\end{example}

Let $\pi$ and $\gamma$ be elements of $\Pi_{r,k}$, define the product in $\MP_{k,r}(x)$ by
\begin{equation}\label{eq:genericproduct}
X_\pi \cdot X_\gamma = \sum_{\nu} a_\nu b_{\nu} (x) X_{\nu|_{[k] \cup [\b{\b{k}}]}}
\end{equation}
where the sum is over all multiset partitions $\nu \in \Gamma_{r,k}$
such that $\nu |_{[k] \cup [\o{k}]} = \pi$ and $\nu|_{[\o{k}] \cup [\b{\b{k}}]} = \o\gamma$
and where, by convention, we consider $\nu|_{[k] \cup [\b{\b{k}}]}$ to be a multiset
partition with entries in $[k] \cup [\o{k}]$ by removing the second bar on elements in $[\b{\b{k}}]$.
By definition, this product is zero if $\pi |_{[\o{k}]} \neq \o{\gamma} |_{[\o{k}]}$
(that is, if the lower row of $\pi$ does not have the same multisets as the upper row of $\o{\gamma}$).

The following example is relatively small, but shows the dependence on the
multiset partitions in $\Gamma_{r,k}$.
\begin{example} Let $\pi = \dcl \dcl 1 \dcr, \dcl 1, \o{1} \dcr, \dcl \o{1} \dcr \dcr \in \Pi_{2,1}$,
then to calculate $X_\pi^2 \in \MP_{2,1}(x)$,
there are four multiset partitions $\nu$ which contribute to the sum with $\nu|_{[1] \cup [\o1]} = \pi$
and $\nu|_{[\o{1}] \cup [\b{\b{1}}]} = \o{\pi}$, they are:
\begin{align*}
&\dcl \dcl 1 \dcr, \dcl \o{1} \dcr, \dcl \b{\b{1}} \dcr, \dcl 1, \o{1}, \b{\b{1}} \dcr \dcr\\
&\dcl \dcl 1, \b{\b{1}} \dcr, \dcl \o{1} \dcr, \dcl 1, \o{1}, \b{\b{1}} \dcr \dcr\\
&\dcl \dcl 1 \dcr, \dcl 1,\o{1} \dcr, \dcl \o{1}, \b{\b{1}} \dcr, \dcl \b{\b{1}} \dcr \dcr\\
&\dcl \dcl 1, \b{\b{1}} \dcr, \dcl 1,\o{1} \dcr, \dcl \o{1}, \b{\b{1}} \dcr \dcr
\end{align*}
We construct the following table where each of the three columns contains the
diagrams representing the three multiset partitions $\nu$ and data representing their
contribution to the sum in Equation \eqref{eq:genericproduct}.
\begin{center}
\begin{tabular}{c|cccc}
three row msp $\nu$&
\begin{tikzpicture}[scale = 0.35,thick, baseline={(0,-1ex/2)}]
\tikzstyle{vertex} = [shape = circle, minimum size = 4pt, inner sep = 1pt]
\node[vertex] (G--2) at (1.5, -2) [shape = circle, draw, fill] {};
\node[vertex] (G--1) at (0.0, -2) [shape = circle, draw, fill] {};
\node[vertex] (G-1) at (0.0, 1) [shape = circle, draw,fill] {};
\node[vertex] (G-2) at (1.5, 1) [shape = circle, draw, fill] {};
\node[left] at (2.1, 1.6) {\tiny $1$};
\node[left] at (0.6, 1.6) {\tiny $1$};
\node[left] at (2.1, -2.7) {\tiny $\b{\b1}$};
\node[left] at (0.6, -2.7) {\tiny $\b{\b1}$};
\node[vertex] (G---2) at (1.5, -0.5) [shape = circle, draw, fill] {};
\node[vertex] (G---1) at (0.0, -0.5) [shape = circle, draw, fill] {};
\node[left] at (2.6, -1.1) {\tiny $\o1$};
\node[left] at (0.6, -1.1) {\tiny $\o1$};
\draw (G-2) .. controls +(0, 0)  .. (G---2);
\draw (G---2) .. controls +(0, 0)  .. (G--2);
\end{tikzpicture}
&
\begin{tikzpicture}[scale = 0.35,thick, baseline={(0,-1ex/2)}]
\tikzstyle{vertex} = [shape = circle, minimum size = 4pt, inner sep = 1pt]
\node[vertex] (G--2) at (1.5, -2) [shape = circle, draw, fill] {};
\node[vertex] (G--1) at (0.0, -2) [shape = circle, draw, fill] {};
\node[vertex] (G-1) at (0.0, 1) [shape = circle, draw,fill] {};
\node[vertex] (G-2) at (1.5, 1) [shape = circle, draw, fill] {};
\node[left] at (2.1, 1.6) {\tiny $1$};
\node[left] at (0.6, 1.6) {\tiny $1$};
\node[left] at (2.1, -2.7) {\tiny $\b{\b1}$};
\node[left] at (0.6, -2.7) {\tiny $\b{\b1}$};
\node[vertex] (G---2) at (1.5, -0.5) [shape = circle, draw, fill] {};
\node[vertex] (G---1) at (0.0, -0.5) [shape = circle, draw, fill] {};
\node[left] at (2.6, -1.1) {\tiny $\o1$};
\node[left] at (0.6, -1.1) {\tiny $\o1$};
\draw (G-2) .. controls +(0, 0)  .. (G---2);
\draw (G---2) .. controls +(0, 0)  .. (G--2);
\draw (G-1) .. controls +(0.8, -1) and +(0.8, 1)  .. (G--1);
\end{tikzpicture}
&
\begin{tikzpicture}[scale = 0.35,thick, baseline={(0,-1ex/2)}]
\tikzstyle{vertex} = [shape = circle, minimum size = 4pt, inner sep = 1pt]
\node[vertex] (G--2) at (1.5, -2) [shape = circle, draw, fill] {};
\node[vertex] (G--1) at (0.0, -2) [shape = circle, draw, fill] {};
\node[vertex] (G-1) at (0.0, 1) [shape = circle, draw,fill] {};
\node[vertex] (G-2) at (1.5, 1) [shape = circle, draw, fill] {};
\node[left] at (2.1, 1.6) {\tiny $1$};
\node[left] at (0.6, 1.6) {\tiny $1$};
\node[left] at (2.1, -2.7) {\tiny $\b{\b1}$};
\node[left] at (0.6, -2.7) {\tiny $\b{\b1}$};
\node[vertex] (G---2) at (1.5, -0.5) [shape = circle, draw, fill] {};
\node[vertex] (G---1) at (0.0, -0.5) [shape = circle, draw, fill] {};
\node[left] at (2.1, -1.1) {\tiny $\o1$};
\node[left] at (0.6, -1.1) {\tiny $\o1$};
\draw (G-2) .. controls +(0, 0)  .. (G---1);
\draw (G--1) .. controls +(0, 0)  .. (G---2);
\end{tikzpicture}
&
\begin{tikzpicture}[scale = 0.35,thick, baseline={(0,-1ex/2)}]
\tikzstyle{vertex} = [shape = circle, minimum size = 4pt, inner sep = 1pt]
\node[vertex] (G--2) at (1.5, -2) [shape = circle, draw, fill] {};
\node[vertex] (G--1) at (0.0, -2) [shape = circle, draw, fill] {};
\node[vertex] (G-1) at (0.0, 1) [shape = circle, draw,fill] {};
\node[vertex] (G-2) at (1.5, 1) [shape = circle, draw, fill] {};
\node[left] at (2.1, 1.6) {\tiny $1$};
\node[left] at (0.6, 1.6) {\tiny $1$};
\node[left] at (2.1, -2.7) {\tiny $\b{\b1}$};
\node[left] at (0.6, -2.7) {\tiny $\b{\b1}$};
\node[vertex] (G---2) at (1.5, -0.5) [shape = circle, draw, fill] {};
\node[vertex] (G---1) at (0.0, -0.5) [shape = circle, draw, fill] {};
\node[left] at (2.1, -1.1) {\tiny $\o1$};
\node[left] at (0.6, -1.1) {\tiny $\o1$};
\draw (G-2) .. controls +(0, 0)  .. (G---1);
\draw (G--1) .. controls +(0, 0)  .. (G---2);
\draw (G-1) .. controls +(0, 0)  .. (G--2);
\end{tikzpicture}
\\\\
$\nu|_{[k] \cup [\b{\b{k}}]}$ and $[\b{\b{k}}] \rightarrow [\o{k}]$&
\begin{tikzpicture}[scale = 0.35,thick, baseline={(0,-1ex/2)}]
\tikzstyle{vertex} = [shape = circle, minimum size = 4pt, inner sep = 1pt]
\node[vertex] (G--2) at (1.5, -1) [shape = circle, draw, fill] {};
\node[vertex] (G--1) at (0.0, -1) [shape = circle, draw, fill] {};
\node[vertex] (G-1) at (0.0, 1) [shape = circle, draw,fill] {};
\node[vertex] (G-2) at (1.5, 1) [shape = circle, draw, fill] {};
\node[left] at (2.1, 1.6) {\tiny $1$};
\node[left] at (0.6, 1.6) {\tiny $1$};
\node[left] at (2.1, -1.6) {\tiny $\o1$};
\node[left] at (0.6, -1.6) {\tiny $\o1$};
\draw (G-2) .. controls +(0, 0) .. (G--2);
\end{tikzpicture}
&
\begin{tikzpicture}[scale = 0.35,thick, baseline={(0,-1ex/2)}]
\tikzstyle{vertex} = [shape = circle, minimum size = 4pt, inner sep = 1pt]
\node[vertex] (G--2) at (1.5, -1) [shape = circle, draw, fill] {};
\node[vertex] (G--1) at (0.0, -1) [shape = circle, draw, fill] {};
\node[vertex] (G-1) at (0.0, 1) [shape = circle, draw,fill] {};
\node[vertex] (G-2) at (1.5, 1) [shape = circle, draw, fill] {};
\node[left] at (2.1, 1.6) {\tiny $1$};
\node[left] at (0.6, 1.6) {\tiny $1$};
\node[left] at (2.1, -1.6) {\tiny $\o1$};
\node[left] at (0.6, -1.6) {\tiny $\o1$};
\draw (G-1) .. controls +(0, 0) .. (G--1);
\draw (G-2) .. controls +(0, 0) .. (G--2);
\end{tikzpicture}
&
\begin{tikzpicture}[scale = 0.35,thick, baseline={(0,-1ex/2)}]
\tikzstyle{vertex} = [shape = circle, minimum size = 4pt, inner sep = 1pt]
\node[vertex] (G--2) at (1.5, -1) [shape = circle, draw, fill] {};
\node[vertex] (G--1) at (0.0, -1) [shape = circle, draw, fill] {};
\node[vertex] (G-1) at (0.0, 1) [shape = circle, draw,fill] {};
\node[vertex] (G-2) at (1.5, 1) [shape = circle, draw, fill] {};
\node[left] at (2.1, 1.6) {\tiny $1$};
\node[left] at (0.6, 1.6) {\tiny $1$};
\node[left] at (2.1, -1.6) {\tiny $\o1$};
\node[left] at (0.6, -1.6) {\tiny $\o1$};
\end{tikzpicture}
&
\begin{tikzpicture}[scale = 0.35,thick, baseline={(0,-1ex/2)}]
\tikzstyle{vertex} = [shape = circle, minimum size = 4pt, inner sep = 1pt]
\node[vertex] (G--2) at (1.5, -1) [shape = circle, draw, fill] {};
\node[vertex] (G--1) at (0.0, -1) [shape = circle, draw, fill] {};
\node[vertex] (G-1) at (0.0, 1) [shape = circle, draw,fill] {};
\node[vertex] (G-2) at (1.5, 1) [shape = circle, draw, fill] {};
\node[left] at (2.1, 1.6) {\tiny $1$};
\node[left] at (0.6, 1.6) {\tiny $1$};
\node[left] at (2.1, -1.6) {\tiny $\o1$};
\node[left] at (0.6, -1.6) {\tiny $\o1$};
\draw (G-2) .. controls +(0, 0) .. (G--2);
\end{tikzpicture}\\\\
coefficient $a_\nu b_\nu(x)$&$x-3$&$2(x-2)$&$4$&$1$
\end{tabular}
\end{center}
Therefore
$$X_\pi^2 = (x-2) X_{\dcl \dcl 1 \dcr, \dcl \o{1} \dcr, \dcl 1, \o{1} \dcr \dcr}
+ 2(x-2) X_{\dcl \dcl 1, \o{1} \dcr, \dcl 1, \o{1} \dcr \dcr}
+ 4 X_{\dcl \dcl 1 \dcr, \dcl 1\dcr, \dcl \o{1} \dcr, \dcl \o{1} \dcr \dcr}~.$$
\end{example}

\begin{example}\label{example:prod}  Let $r=4$, $k=2$ and consider the
multiset partitions $\pi= \dcl\dcl1,1\dcr,$ $\dcl\o1,\o2\dcr,$ $\dcl1,2,\o1,\o2\dcr\dcr$,
and $\gamma = \dcl\dcl1,2\dcr,$ $\dcl\o1,\o1\dcr,$ $\dcl1,2,\o1,\o1\dcr\dcr$.
There are four multiset partitions $\nu$ with entries in $\{1,2,\o1,\o2,\b{\b1},\b{\b2}\}$
which contribute to the terms in the product $X_\pi X_\gamma$.

\begin{center}
\begin{tabular}{c|cccc}
three row msp $\nu$&
\begin{tikzpicture}[scale = 0.35,thick, baseline={(0,-1ex/2)}]
\tikzstyle{vertex} = [shape = circle, minimum size = 4pt, inner sep = 1pt]
\node[vertex] (G--4) at (4.5, -2) [shape = circle, draw, fill] {};
\node[vertex] (G--3) at (3.0, -2) [shape = circle, draw, fill] {};
\node[vertex] (G-3) at (3.0, 1) [shape = circle, draw, fill] {};
\node[vertex] (G-4) at (4.5, 1) [shape = circle, draw, fill] {};
\node[vertex] (G--2) at (1.5, -2) [shape = circle, draw, fill] {};
\node[vertex] (G--1) at (0.0, -2) [shape = circle, draw, fill] {};
\node[vertex] (G-1) at (0.0, 1) [shape = circle, draw,fill] {};
\node[vertex] (G-2) at (1.5, 1) [shape = circle, draw, fill] {};
\node[left] at (5.1, 1.6) {\tiny $2$};
\node[left] at (3.6, 1.6) {\tiny $1$};
\node[left] at (2.1, 1.6) {\tiny $1$};
\node[left] at (0.6, 1.6) {\tiny $1$};
\node[left] at (5.1, -2.7) {\tiny $\b{\b1}$};
\node[left] at (3.6, -2.7) {\tiny $\b{\b1}$};
\node[left] at (2.1, -2.7) {\tiny $\b{\b1}$};
\node[left] at (0.6, -2.7) {\tiny $\b{\b1}$};
\draw (G-3) .. controls +(0.5, -0.5) and +(-0.5, -0.5) .. (G-4);
\draw (G-4) .. controls +(0, -1) and +(0, 1) .. (G--4);
\draw (G--4) .. controls +(-0.5, 0.5) and +(0.5, 0.5) .. (G--3);
\draw (G--2) .. controls +(-0.5, 0.5) and +(0.5, 0.5) .. (G--1);
\draw (G-1) .. controls +(0.5, -0.5) and +(-0.5, -0.5) .. (G-2);
\node[vertex] (G---4) at (4.5, -0.5) [shape = circle, draw, fill] {};
\node[vertex] (G---3) at (3.0, -0.5) [shape = circle, draw, fill] {};
\node[vertex] (G---2) at (1.5, -0.5) [shape = circle, draw, fill] {};
\node[vertex] (G---1) at (0.0, -0.5) [shape = circle, draw, fill] {};
\node[left] at (4.8, -1.1) {\tiny $\o2$};
\node[left] at (3.6, -1.1) {\tiny $\o1$};
\node[left] at (2.1, -1.1) {\tiny $\o2$};
\node[left] at (0.6, -1.1) {\tiny $\o1$};
\draw (G---4) .. controls +(-.5, .5)  and +(.5, -.5) .. (G---3);
\draw (G---2) .. controls +(-.5, .5)  and +(.5, -.5) .. (G---1);
\end{tikzpicture}
&
\begin{tikzpicture}[scale = 0.35,thick, baseline={(0,-1ex/2)}]
\tikzstyle{vertex} = [shape = circle, minimum size = 4pt, inner sep = 1pt]
\node[vertex] (G--4) at (4.5, -2) [shape = circle, draw, fill] {};
\node[vertex] (G--3) at (3.0, -2) [shape = circle, draw, fill] {};
\node[vertex] (G-3) at (3.0, 1) [shape = circle, draw, fill] {};
\node[vertex] (G-4) at (4.5, 1) [shape = circle, draw, fill] {};
\node[vertex] (G--2) at (1.5, -2) [shape = circle, draw, fill] {};
\node[vertex] (G--1) at (0.0, -2) [shape = circle, draw, fill] {};
\node[vertex] (G-1) at (0.0, 1) [shape = circle, draw,fill] {};
\node[vertex] (G-2) at (1.5, 1) [shape = circle, draw, fill] {};
\node[left] at (5.1, 1.6) {\tiny $2$};
\node[left] at (3.6, 1.6) {\tiny $1$};
\node[left] at (2.1, 1.6) {\tiny $1$};
\node[left] at (0.6, 1.6) {\tiny $1$};
\node[left] at (5.1, -2.7) {\tiny $\b{\b1}$};
\node[left] at (3.6, -2.7) {\tiny $\b{\b1}$};
\node[left] at (2.1, -2.7) {\tiny $\b{\b1}$};
\node[left] at (0.6, -2.7) {\tiny $\b{\b1}$};
\draw (G-3) .. controls +(0.5, -0.5) and +(-0.5, -0.5) .. (G-4);
\draw (G-4) .. controls +(0, -1) and +(0, 1) .. (G--4);
\draw (G--4) .. controls +(-0.5, 0.5) and +(0.5, 0.5) .. (G--3);
\draw (G--2) .. controls +(-0.5, 0.5) and +(0.5, 0.5) .. (G--1);
\draw (G-1) .. controls +(0.5, -0.5) and +(-0.5, -0.5) .. (G-2);
\node[vertex] (G---4) at (4.5, -0.5) [shape = circle, draw, fill] {};
\node[vertex] (G---3) at (3.0, -0.5) [shape = circle, draw, fill] {};
\node[vertex] (G---2) at (1.5, -0.5) [shape = circle, draw, fill] {};
\node[vertex] (G---1) at (0.0, -0.5) [shape = circle, draw, fill] {};
\node[left] at (4.8, -1.1) {\tiny $\o2$};
\node[left] at (3.6, -1.1) {\tiny $\o1$};
\node[left] at (2.1, -1.1) {\tiny $\o2$};
\node[left] at (0.6, -1.1) {\tiny $\o1$};
\draw (G---4) .. controls +(-.5, .5)  and +(.5, -.5) .. (G---3);
\draw (G---2) .. controls +(-.5, .5)  and +(.5, -.5) .. (G---1);
\draw (G-2) .. controls +(0.5, -0.5) and +(0.5, 0.5) .. (G--2);
\end{tikzpicture}
&
\begin{tikzpicture}[scale = 0.35,thick, baseline={(0,-1ex/2)}]
\tikzstyle{vertex} = [shape = circle, minimum size = 4pt, inner sep = 1pt]
\node[vertex] (G--4) at (4.5, -2) [shape = circle, draw, fill] {};
\node[vertex] (G--3) at (3.0, -2) [shape = circle, draw, fill] {};
\node[vertex] (G-3) at (3.0, 1) [shape = circle, draw, fill] {};
\node[vertex] (G-4) at (4.5, 1) [shape = circle, draw, fill] {};
\node[vertex] (G--2) at (1.5, -2) [shape = circle, draw, fill] {};
\node[vertex] (G--1) at (0.0, -2) [shape = circle, draw, fill] {};
\node[vertex] (G-1) at (0.0, 1) [shape = circle, draw,fill] {};
\node[vertex] (G-2) at (1.5, 1) [shape = circle, draw, fill] {};
\node[left] at (5.1, 1.6) {\tiny $2$};
\node[left] at (3.6, 1.6) {\tiny $1$};
\node[left] at (2.1, 1.6) {\tiny $1$};
\node[left] at (0.6, 1.6) {\tiny $1$};
\node[left] at (5.1, -2.7) {\tiny $\b{\b1}$};
\node[left] at (3.6, -2.7) {\tiny $\b{\b1}$};
\node[left] at (2.1, -2.7) {\tiny $\b{\b1}$};
\node[left] at (0.6, -2.7) {\tiny $\b{\b1}$};
\draw (G-3) .. controls +(0.5, -0.5) and +(-0.5, -0.5) .. (G-4);
\draw (G-4) .. controls +(0, -1) and +(0, 1) .. (G---4);
\draw (G--4) .. controls +(-0.5, 0.5) and +(0.5, 0.5) .. (G--3);
\draw (G--2) .. controls +(-0.5, 0.5) and +(0.5, 0.5) .. (G--1);
\draw (G-1) .. controls +(0.5, -0.5) and +(-0.5, -0.5) .. (G-2);
\node[vertex] (G---4) at (4.5, -0.5) [shape = circle, draw, fill] {};
\node[vertex] (G---3) at (3.0, -0.5) [shape = circle, draw, fill] {};
\node[vertex] (G---2) at (1.5, -0.5) [shape = circle, draw, fill] {};
\node[vertex] (G---1) at (0.0, -0.5) [shape = circle, draw, fill] {};
\node[left] at (5.1, -1.1) {\tiny $\o2$};
\node[left] at (3.6, -1.1) {\tiny $\o1$};
\node[left] at (1.8, -1.1) {\tiny $\o2$};
\node[left] at (0.6, -1.1) {\tiny $\o1$};
\draw (G---4) .. controls +(-.5, .5)  and +(.5, -.5) .. (G---3);
\draw (G---2) .. controls +(-.5, .5)  and +(.5, -.5) .. (G---1);
\draw (G--2) .. controls +(0, 0) .. (G---2);
\draw (G-2) .. controls +(0, 0) .. (G--3);
\end{tikzpicture}
&
\begin{tikzpicture}[scale = 0.35,thick, baseline={(0,-1ex/2)}]
\tikzstyle{vertex} = [shape = circle, minimum size = 4pt, inner sep = 1pt]
\node[vertex] (G--4) at (4.5, -2) [shape = circle, draw, fill] {};
\node[vertex] (G--3) at (3.0, -2) [shape = circle, draw, fill] {};
\node[vertex] (G-3) at (3.0, 1) [shape = circle, draw, fill] {};
\node[vertex] (G-4) at (4.5, 1) [shape = circle, draw, fill] {};
\node[vertex] (G--2) at (1.5, -2) [shape = circle, draw, fill] {};
\node[vertex] (G--1) at (0.0, -2) [shape = circle, draw, fill] {};
\node[vertex] (G-1) at (0.0, 1) [shape = circle, draw,fill] {};
\node[vertex] (G-2) at (1.5, 1) [shape = circle, draw, fill] {};
\node[vertex] (G---4) at (4.5, -0.5) [shape = circle, draw, fill] {};
\node[vertex] (G---3) at (3.0, -0.5) [shape = circle, draw, fill] {};
\node[vertex] (G---2) at (1.5, -0.5) [shape = circle, draw, fill] {};
\node[vertex] (G---1) at (0.0, -0.5) [shape = circle, draw, fill] {};
\node[left] at (5.1, 1.6) {\tiny $2$};
\node[left] at (3.6, 1.6) {\tiny $1$};
\node[left] at (2.1, 1.6) {\tiny $1$};
\node[left] at (0.6, 1.6) {\tiny $1$};
\node[left] at (5.1, -2.7) {\tiny $\b{\b1}$};
\node[left] at (3.6, -2.7) {\tiny $\b{\b1}$};
\node[left] at (2.1, -2.7) {\tiny $\b{\b1}$};
\node[left] at (0.6, -2.7) {\tiny $\b{\b1}$};
\draw (G-3) .. controls +(0.5, -0.5) and +(-0.5, -0.5) .. (G-4);
\draw (G-4) .. controls +(0, -1) and +(0, 1) .. (G---4);
\draw (G--4) .. controls +(-0.5, 0.5) and +(0.5, 0.5) .. (G--3);
\draw (G--2) .. controls +(-0.5, 0.5) and +(0.5, 0.5) .. (G--1);
\draw (G-1) .. controls +(0.5, -0.5) and +(-0.5, -0.5) .. (G-2);
\node[left] at (5.1, -1.1) {\tiny $\o2$};
\node[left] at (3.6, -1.1) {\tiny $\o1$};
\node[left] at (1.8, -1.1) {\tiny $\o2$};
\node[left] at (0.6, -1.1) {\tiny $\o1$};
\draw (G---4) .. controls +(-.5, .5)  and +(.5, -.5) .. (G---3);
\draw (G---2) .. controls +(-.5, .5)  and +(.5, -.5) .. (G---1);
\draw (G--2) .. controls +(0, 0) .. (G---2);
\end{tikzpicture}
\\\\
$\nu|_{[k] \cup [\b{\b{k}}]}$ and $[\b{\b{k}}] \rightarrow [\o{k}]$&
\begin{tikzpicture}[scale = 0.35,thick, baseline={(0,-1ex/2)}]
\tikzstyle{vertex} = [shape = circle, minimum size = 4pt, inner sep = 1pt]
\node[vertex] (G--4) at (4.5, -1) [shape = circle, draw, fill] {};
\node[vertex] (G--3) at (3.0, -1) [shape = circle, draw, fill] {};
\node[vertex] (G-3) at (3.0, 1) [shape = circle, draw, fill] {};
\node[vertex] (G-4) at (4.5, 1) [shape = circle, draw, fill] {};
\node[vertex] (G--2) at (1.5, -1) [shape = circle, draw, fill] {};
\node[vertex] (G--1) at (0.0, -1) [shape = circle, draw, fill] {};
\node[vertex] (G-1) at (0.0, 1) [shape = circle, draw,fill] {};
\node[vertex] (G-2) at (1.5, 1) [shape = circle, draw, fill] {};
\node[left] at (5.1, 1.6) {\tiny $2$};
\node[left] at (3.6, 1.6) {\tiny $1$};
\node[left] at (2.1, 1.6) {\tiny $1$};
\node[left] at (0.6, 1.6) {\tiny $1$};
\node[left] at (5.1, -1.6) {\tiny $\o1$};
\node[left] at (3.6, -1.6) {\tiny $\o1$};
\node[left] at (2.1, -1.6) {\tiny $\o1$};
\node[left] at (0.6, -1.6) {\tiny $\o1$};
\draw (G-3) .. controls +(0.5, -0.5) and +(-0.5, -0.5) .. (G-4);
\draw (G-4) .. controls +(0, -1) and +(0, 1) .. (G--4);
\draw (G--4) .. controls +(-0.5, 0.5) and +(0.5, 0.5) .. (G--3);
\draw (G--2) .. controls +(-0.5, 0.5) and +(0.5, 0.5) .. (G--1);
\draw (G-1) .. controls +(0.5, -0.5) and +(-0.5, -0.5) .. (G-2);
\end{tikzpicture}
& 
\begin{tikzpicture}[scale = 0.35,thick, baseline={(0,-1ex/2)}]
\tikzstyle{vertex} = [shape = circle, minimum size = 4pt, inner sep = 1pt]
\node[vertex] (G--4) at (4.5, -1) [shape = circle, draw,fill] {};
\node[vertex] (G--3) at (3.0, -1) [shape = circle, draw, fill] {};
\node[vertex] (G-3) at (3.0, 1) [shape = circle, draw, fill] {};
\node[vertex] (G-4) at (4.5, 1) [shape = circle, draw, fill] {};
\node[vertex] (G--2) at (1.5, -1) [shape = circle, draw, fill] {};
\node[vertex] (G--1) at (0.0, -1) [shape = circle, draw, fill] {};
\node[vertex] (G-1) at (0.0, 1) [shape = circle, draw, fill] {};
\node[vertex] (G-2) at (1.5, 1) [shape = circle, draw,fill] {};
\node[left] at (5.1, 1.6) {\tiny $2$};
\node[left] at (3.6, 1.6) {\tiny $1$};
\node[left] at (2.1, 1.6) {\tiny $1$};
\node[left] at (0.6, 1.6) {\tiny $1$};
\node[left] at (5.1, -1.6) {\tiny $\o1$};
\node[left] at (3.6, -1.6) {\tiny $\o1$};
\node[left] at (2.1, -1.6) {\tiny $\o1$};
\node[left] at (0.6, -1.6) {\tiny $\o1$};
\draw (G-3) .. controls +(0.5, -0.5) and +(-0.5, -0.5) .. (G-4);
\draw (G-4) .. controls +(0, -1) and +(0, 1) .. (G--4);
\draw (G--4) .. controls +(-0.5, 0.5) and +(0.5, 0.5) .. (G--3);
\draw (G--2) .. controls +(0, 1) and +(0, -1) .. (G-2);
\draw (G--2) .. controls +(-0.5, 0.5) and +(0.5, 0.5) .. (G--1);
\draw (G-1) .. controls +(0.5, -0.5) and +(-0.5, -0.5) .. (G-2);
\end{tikzpicture}
&
\begin{tikzpicture}[scale = 0.35,thick, baseline={(0,-1ex/2)}]
\tikzstyle{vertex} = [shape = circle, minimum size = 4pt, inner sep = 1pt]
\node[vertex] (G--4) at (4.5, -1) [shape = circle, draw,fill] {};
\node[vertex] (G--3) at (3.0, -1) [shape = circle, draw, fill] {};
\node[vertex] (G-3) at (3.0, 1) [shape = circle, draw, fill] {};
\node[vertex] (G-4) at (4.5, 1) [shape = circle, draw, fill] {};
\node[vertex] (G--2) at (1.5, -1) [shape = circle, draw,fill] {};
\node[vertex] (G--1) at (0.0, -1) [shape = circle, draw,fill] {};
\node[vertex] (G-1) at (0.0, 1) [shape = circle, draw,fill] {};
\node[vertex] (G-2) at (1.5, 1) [shape = circle, draw,fill] {};
\node[left] at (5.1, 1.6) {\tiny $2$};
\node[left] at (3.6, 1.6) {\tiny $1$};
\node[left] at (2.1, 1.6) {\tiny $1$};
\node[left] at (0.6, 1.6) {\tiny $1$};
\node[left] at (5.1, -1.6) {\tiny $\o1$};
\node[left] at (3.6, -1.6) {\tiny $\o1$};
\node[left] at (2.1, -1.6) {\tiny $\o1$};
\node[left] at (0.6, -1.6) {\tiny $\o1$};
\draw (G-3) .. controls +(0.5, -0.5) and +(-0.5, -0.5) .. (G-4);
\draw (G--4) .. controls +(-0.5, 0.5) and +(0.5, 0.5) .. (G--3);
\draw (G--2) .. controls +(0, 1) and +(0, -1) .. (G-2);
\draw (G--2) .. controls +(-0.5, 0.5) and +(0.5, 0.5) .. (G--1);
\draw (G-1) .. controls +(0.5, -0.5) and +(-0.5, -0.5) .. (G-2);
\end{tikzpicture}
& 
\begin{tikzpicture}[scale = 0.35,thick, baseline={(0,-1ex/2)}]
\tikzstyle{vertex} = [shape = circle, minimum size = 4pt, inner sep = 1pt]
\node[vertex] (G--4) at (4.5, -1) [shape = circle, draw,fill] {};
\node[vertex] (G--3) at (3.0, -1) [shape = circle, draw,fill] {};
\node[vertex] (G-3) at (3.0, 1) [shape = circle, draw,fill] {};
\node[vertex] (G-4) at (4.5, 1) [shape = circle, draw, fill] {};
\node[vertex] (G--2) at (1.5, -1) [shape = circle, draw,fill] {};
\node[vertex] (G--1) at (0.0, -1) [shape = circle, draw,fill] {};
\node[vertex] (G-1) at (0.0, 1) [shape = circle, draw, fill] {};
\node[vertex] (G-2) at (1.5, 1) [shape = circle, draw,fill] {};
\node[left] at (5.1, 1.6) {\tiny $2$};
\node[left] at (3.6, 1.6) {\tiny $1$};
\node[left] at (2.1, 1.6) {\tiny $1$};
\node[left] at (0.6, 1.6) {\tiny $1$};
\node[left] at (5.1, -1.6) {\tiny $\o1$};
\node[left] at (3.6, -1.6) {\tiny $\o1$};
\node[left] at (2.1, -1.6) {\tiny $\o1$};
\node[left] at (0.6, -1.6) {\tiny $\o1$};
\draw (G-3) .. controls +(0.5, -0.5) and +(-0.5, -0.5) .. (G-4);
\draw (G--4) .. controls +(-0.5, 0.5) and +(0.5, 0.5) .. (G--3);
\draw (G--2) .. controls +(-0.5, 0.5) and +(0.5, 0.5) .. (G--1);
\draw (G-1) .. controls +(0.5, -0.5) and +(-0.5, -0.5) .. (G-2);
\end{tikzpicture}
\\\\
coefficient $a_\nu b_\nu(x)$&
$(x-3)$&$(x-2)$&$1$&$2$
\end{tabular}
\end{center}
Therefore
\begin{align*}
X_\pi X_\gamma = & (x-3) X_{\dcl \dcl 1,1 \dcr, \dcl \o{1}, \o{1} \dcr, \dcl 1, 2, \o{1}, \o{1} \dcr \dcr}
+ (x-2) X_{\dcl \dcl 1, 1, \o{1} ,\o{1}\dcr, \dcl 1, 2, \o{1}, \o{1}\dcr \dcr}
 +  X_{\dcl \dcl 1, 1, \o{1}, \o{1} \dcr, \dcl 1, 2\dcr, \dcl \o{1}, \o{1} \dcr \dcr} \\
&+  2 X_{\dcl \dcl 1, 1\dcr ,\dcl  \o{1}, \o{1} \dcr, \dcl 1, 2\dcr,  \dcl \o{1}, \o{1} \dcr \dcr}~.
\end{align*}
\end{example}
Observe that by definition this product is a well-defined algebra product.  In addition, for any $\pi, \gamma, \tau\in \Pi_{r,k}$
the coefficient of $X_\tau$ in a product $X_\pi X_\gamma$ is always a polynomial in the variable $x$.
\begin{prop} \label{prop:associative}
The product defined in Equation \eqref{eq:genericproduct}
is associative.
\end{prop}
\begin{proof}
Proving associativity of this product directly is difficult since we would need to
give an interpretation to a product of three basis elements. That is, we would have
to find a formula for the coefficients occurring in these triple products.   Thus, we show associativity
indirectly using the fact that the algebra $\MP_{r,k}(x)$ is isomorphic to an endomorphism
algebra $\mathcal{A}_{r,k}(n)$ (see Section \ref{sec:centralizer} for details), for specialized values of $x=n$ with $n$ and integer
such that $n\geq 2r$, see Theorem \ref{th:isomorphism}.   The algebra $\mathcal{A}_{r,k}(n)$ has a basis $\{O_\pi\, |\, \pi \in \Pi_{r,k}\}$ and
in Theorem \ref{th:isomorphism} we show that $X_\pi$ corresponds to $O_\pi$ under the isomorphism.

For any
$\pi, \gamma, \zeta, \tau \in \Pi_{r,k}$,
the coefficient of $X_\tau$ in $(X_\pi X_\gamma) X_\zeta-X_\pi(X_\gamma X_\zeta)$
is a polynomial in $x$.  For all $x = n$ where $n$ is an integer which is greater than or equal to
$2r$, by Theorem \ref{th:isomorphism} this polynomial evaluates to $0$ because
\begin{align*}
\hbox{the coefficient of }&X_\tau\hbox{ in }(X_\pi X_\gamma) X_\zeta-X_\pi(X_\gamma X_\zeta)\\
&= \hbox{the coefficient of }O_\tau\hbox{ in }(O_\pi O_\gamma) O_\zeta-O_\pi(O_\gamma O_\zeta)
\end{align*}
and because the product of the $O_\pi$ is associative in $\mathcal{A}_{r,k}(n)$.   Since this holds for an infinite
number of values if $x=n$, as a polynomial in $x$ the coefficient must be the $0$ polynomial.
\end{proof}

A multiset of $[k] \cup [\o{k}]$ is self-symmetric if replacing $i$ with $\o{i}$
and $\o{i}$ with $i$ is the same multiset.
A multiset partition is self-symmetric if each of the parts
are individually self-symmetric.  For example $\dcl \dcl 1,1,\o1,\o1 \dcr, \dcl 1,2,2,\o1,\o2,\o2\dcr,
\dcl 1, \o1 \dcr, \dcl 2,3, \o2, \o3\dcr\dcr$ is self-symmetric, but for instance,
$\dcl \dcl 1, 2, \o2 \dcr, \dcl 2, \o1, \o2 \dcr \dcr$ is not.

\begin{prop} For $r,k > 0$, the element $I_{r,k} = \sum_{\pi} X_\pi$ where the sum is over
all self-symmetric multiset partitions $\pi \in \Pi_{r,k}$
is the identity element of $\MP_{r,k}(x)$.
\end{prop}

\begin{proof}
Fix a $\gamma \in \Pi_{r,k}$ and to compute $I_{r,k} X_\gamma$ we know that
for any self symmetric $\pi$, $X_\pi X_\gamma = 0$
unless $\pi|_{[\o{k}]} = \o{\gamma} |_{[\o{k}]}$.
Therefore $I_{r,k} X_\gamma = X_\pi X_\gamma$ where $\pi$ is the
self-symmetric element of $\Pi_{r,k}$ such that $\pi|_{[\o{k}]} = \o{\gamma} |_{[\o{k}]}$.

The product $X_\pi X_\gamma$ is the sum over all $\nu$ such that
$\nu |_{[k] \cup [\o{k}]} = \pi$ and $\nu|_{[\o{k}] \cup [\b{\b{k}}]} = \o\gamma$.
Notice that if we unbar the multisets in $\nu|_{[\o{k}]}$, the result is the multiset
partition  $\nu|_{[k]}$ it follows that $\nu |_{[k] \cup [\b{\b{k}}]} = \gamma$ (if we remove one
bar from the double barred elements).  Thus, it follows that
$\nu = \dcl S \uplus \dcl i : \o{i} \in S \dcr : S \in \o{\gamma} \dcr$ is the single
$\nu \in \Gamma_{r,k}$ that satisfies this condition.

The multiset $\dcl S \in \nu : \forall i \in S, i \in [\o{k}] \dcr$
is empty since $\pi$ is self-symmetric so each multiset with an $\o{i} \in [\o{k}]$ also contains
the corresponding $i \in [k]$, therefore $b_\nu(x) = 1$ by Equation \eqref{eq:defbnu}.
For each multiset $S$ in $\nu|_{[k] \cup [\b{\b{k}}]}$, the multisets in $\nu_S$
will consist of only the multiset $S \uplus \dcl \o{i} : i \in S|_{[k]} \dcr$ (possibly with a multiplicity)
and hence $\ell(\nu_S)! = m(\nu_S)!$ and by Equation \eqref{eq:defanu}, $a_\nu = 1$. It follows
that $I_{r,k} X_\gamma = X_\pi X_\gamma = X_\gamma$.

Since the transformation which sends $X_\pi \mapsto X_{\o{\pi}}$ is
an algebra antihomomorphism, a right identity exists
and the right and left identity of an algebra must be the same.
\end{proof}

\section{ $\MP_{r,k}(n)$ is a centralizer algebra of $\CC S_n$} \label{sec:centralizer}
In this section we introduce a centralizer algebra of the
symmetric group.  Our objective is to show that this algebra is
isomorphic to  $\MP_{r,k}(n)$.

\subsection{An $S_n$ action on polynomials}
Let $V_{n,k} = (\mathbb{C}^n)^k$ denote the $nk$ dimensional vector space of
sequences $(v_1, v_2, \ldots, v_k)$ where $v_j \in \mathbb{C}^n$.
Notice that $V_{n,k}$ can be identified with the space of $n\times k$
matrices $M_{n,k}$, where each $v_j$ is a column vector.
With this identification, we can define $e_{ij}$ as
the $n\times k$ matrix with 1 in the $ij$-the entry and zeros
everywhere else, then $\{e_{ij}\, |\, 1\leq i\leq n, 1\leq j\leq k\}$ is a basis for $V_{n,k}$.

The entries $x_{ij}$ of a matrix $X$ in $M_{n,k}$ can be viewed as functions on $V_{n,k}$.
Then  $\bP(V_{n,k})$ is the polynomial ring in the commuting variables
\[ \{ x_{ij}\, |\,  1\leq i\leq n, 1\leq j\leq k\}~.\]
We will denote by $\bP^r(V_{n,k})$ the space of polynomials of degree $r$
(the symmetric tensor).
This space has as basis the monomials in the $x_{ij}$ of degree $r$, we order the monomials
in weakly decreasing order in such a way that $j_1\leq j_2\leq \ldots \leq j_r$ and $i_s\leq i_t$ if $j_s=j_t$.
Assuming this order, then
\[
\{  x_{i_1j_1} x_{i_2j_2}\cdots x_{i_rj_r}\, |\, 1\leq i_s\leq n \text{ and } 1\leq j_t\leq k \}
\]
is an ordered basis for $\bP^r(V_{n,k})$.  Thus, the dimension of $\bP^r(V_{n,k})$ is ${nk +r -1 \choose r}$.

Using the identification of $V_{n,k}$ with the $n\times k$ matrices, then $GL_n$ acts on $V_{n,k}$ via
left multiplication.  This action induces an action of $GL_n$ on $\bP(V_{n,k})$, for the details of this
action see \cite{GoodWall2}, Section 5.6.2.  Now the symmetric group, $S_n$, embeds in $GL_n$ as
permutation matrices and thus acts on $V_{n,k}$; for $\sigma \in S_n$ and $e_{ij}$ a basis element
of $V_{n,k}$, we have
\[ \sigma \cdot e_{i,j} = e_{\sigma(i),j}. \]
With the identification of $e_{ij}$ as a matrix unit, this action corresponds to permutation of rows of the
matrix or equivalently left multiplication by the permutation matrix $\sigma$.  Then by restriction,
$S_n$ acts on $\bP(V_{n,k})$, this action of $S_n$ has a
simple description on monomials.  For $ x_{i_1j_1} x_{i_2j_2}\cdots x_{i_rj_r}\in \bP^r(V_{n,k})$ and $\sigma \in S_n$,
we have
\begin{equation}\label{action}
\sigma \cdot  x_{i_1j_1} x_{i_2j_2}\cdots x_{i_rj_r} =  x_{\sigma^{-1}(i_1)j_1} x_{\sigma^{-1}(i_2)j_2}\cdots x_{\sigma^{-1}(i_r)j_r}~.
\end{equation}

\begin{remark}
This action decomposes $\bP^r(V_{n,k})$ into invariant subspaces indexed by multisets of size $r$ with elements from the set
$\{1, \dots, k\}$.
These multisets can be identified with a weakly increasing sequence $(j_1, j_2, \ldots, j_r)$
which correspond to the second index of the variables $x_{ij}$.
In particular, the subspace of monomials with $j_s = s$ can be identified with the tensor product
$(\mathbb{C}^n)^{\otimes r}$
and the action of $S_n$ is the diagonal action on tensor space.  Recall that the centralizer of the diagonal
action of $S_n$ on tensor space is isomorphic to the partition algebra when $n\geq 2r$ \cite{Jones}.
 \end{remark}

In Section 5, when we describe the branching rule and the restriction
to simplify the presentation we will use the
following notation.  For a subset $S \subseteq \{1,2,\ldots,k\}$, we will denote $V_{n,S} \subseteq V_{n,k}$
where $V_{n,S}$ is the span of elements $(v_1, v_2, \ldots, v_k)$ where $v_j = 0$ if
$j \notin S$.  That is, $V_{n,\{1,2,\ldots, k\}} = V_{n,k}$ and
$V_{n,S}$ is the span of the basis elements $\{ e_{ij} : 1 \leq i \leq n, j \in S\}$.

\subsection{The centralizer of the action}
In this section we are interested in the algebra, $\AA_{r,k}(n):=\End_{\CC S_n}(\bP^r(V_{n,k}))$,
of endomorphisms that commute with the action of $S_n$ on $\bP^r(V_{n,k})$, see Equation \eqref{action}.   That is,
\[ \AA_{r,k}(n) = \{ f: \bP^r(V_{n,k})\rightarrow \bP^r(V_{n,k})\, |\, f\sigma = \sigma f \text{ for all } \sigma \in S_n\}. \]
First we  want to describe a basis for this algebra.
In order to do this, we use the fact that an operator $f\in \End(\bP^r(V_{n,k}))$ is determined by a
 ${nk+r-1\choose r}\times {nk+r-1\choose r}$ matrix $A=(A_{(\mathbf{i}, \mathbf{j})}^{(\mathbf{i}', \mathbf{j}')})$,
 where $(\mathbf{i}, \mathbf{j})=((i_1, j_1), \dots, (i_r,j_r))$
satisfying  $i_s\in[n]$ and $j_s\in[k]$ ordered such that $j_1\leq j_2\leq \ldots \leq j_r$ and $i_s\leq i_t$ if $j_s=j_t$.
Similarly  $(\mathbf{i}',\mathbf{j}') = ((i_1', \o{j}_1), \ldots,  (i_r', \o{j}_r) )$,
where $i'_a\in [n]$ and $\overline{j}_b \in [\o k]$.
The reason we place an overline is to make sure that we emphasize
that the primed correspond to the column indices. The order for the column
indices is the same as for the unbarred.  Abusing notation,  we will use
$\mathbf{j}$ as both a sequence ordered in weakly increasing order and as a
multiset of the set $[k]$ with $r$ elements.   Similarly, $\mathbf{j}'$
denotes a multiset of $[\o k]$ and a weakly increasing sequence.

To simplify notation we set $x_{(\mathbf{i}, \mathbf{j})} = x_{i_1j_1} x_{i_2j_2}\cdots x_{i_rj_r}$, in this notation, $\sigma x_{(\mathbf{i}, \mathbf{j})} = x_{(\sigma^{-1}(\mathbf{i}), \mathbf{j})}$
where for $\sigma \in S_n$, $\sigma^{-1}(\mathbf{i}) = (\sigma^{-1}(i_1),\sigma^{-1}(i_1), \ldots, \sigma^{-1}(i_r))$.  Therefore, if $A$  is a matrix corresponding to an endomorphism in  $\End (\bP^r(V_{n,k}))$, then
\[ A x_{(\mathbf{i}, \mathbf{j})} = \sum_{(\mathbf{i}', \mathbf{j}')} A_{(\mathbf{i}, \mathbf{j})}^{(\mathbf{i}', \mathbf{j}')} x_{(\mathbf{i}', \mathbf{j}')}~. \]

\begin{lemma}\label{lem:elements} For $r,k,n>0$, the element
$A\in \AA_{r,k}(n)$  if and only if $A_{(\mathbf{i},\mathbf{j})}^{(\mathbf{i}', \mathbf{j}')} = A_{(\sigma(\mathbf{i}), \mathbf{j})}^{(\sigma(\mathbf{i}'), \mathbf{j}')}$
for all $\sigma\in S_n$ and ${(\mathbf{i}, \mathbf{j})}$ and  ${(\mathbf{i}', \mathbf{j}')}$ are as specified above.
\end{lemma}
\begin{proof}
A straightforward computation show that
$\sigma A  x_{(\mathbf{i}, \mathbf{j})} =  \sum_{(\mathbf{i}', \mathbf{j}')} A_{(\mathbf{i}, \mathbf{j})}^{(\mathbf{i}', \mathbf{j}')} x_{(\sigma^{-1}(\mathbf{i}'), \mathbf{j}')}$ and
$ A \sigma x_{(\mathbf{i}, \mathbf{j})} = \sum_{(\mathbf{i}'', \mathbf{j}'')} A_{(\sigma^{-1}(\mathbf{i}), \mathbf{j})}^{(\mathbf{i}'', \mathbf{j}'')} x_{(\mathbf{i}'', \mathbf{j}'')}$.
Since $A\in \AA_{r,k}(n)$, $A\sigma = \sigma A$.  Hence, by equating coefficients, the claim follows.
\qedhere
\end{proof}
The main consequence of Lemma \ref{lem:elements} is that $A$ commutes with the action of $S_n$ on $\bP^r(V_{n,k})$ if and only
if the matrix entries of $A$ are equal on $S_n$-orbits.
We now describe these orbits and relate them to the combinatorics described in the previous sections.

Observe that each fixed pair $((\mathbf{i},\mathbf{j}), (\mathbf{i}',\mathbf{j}'))$ determines a matrix unit
$E_{(\mathbf{i},\mathbf{j})}^{(\mathbf{i}',\mathbf{j}')}\in \End(\bP^r(V_{n,k}))$ with 1 in the
$((\mathbf{i},\mathbf{j}), (\mathbf{i}',\mathbf{j}'))$ position and zeros everywhere else.
Therefore, the elements in $\AA_{r,k}(n)$ can be written as linear combinations of these matrix units.
Each fixed pair $((\mathbf{i},\mathbf{j}), (\mathbf{i}',\mathbf{j}'))$ determines a partition of
the multiset $\mathbf{j} \uplus \mathbf{j}'$ into at most $n$ blocks obtained by placing
$j_a$ and $j_b$ in the same block if and only if $i_a=i_b$.
Two pairs $((\mathbf{i},\mathbf{j}), (\mathbf{i}',\mathbf{j}'))$
and $((\mathbf{s},\mathbf{t}), (\mathbf{s'},\mathbf{t'}))$ are in the same $S_n$ orbit if they give rise to the same multiset
partition.

Recall that in Section \ref{ssec:coloredmsp} we introduced the notion of a colored multiset partition.
This is a data structure that encodes the objects which index the matrix units in $\End(\bP^r(V_{n,k}))$.
Each pair $((\mathbf{i},\mathbf{j}), (\mathbf{i}',\mathbf{j}'))$
corresponds to a colored multiset partition, where the colors are determined by the $\mathbf{i}$ and $\mathbf{i}'$ and the
underlying multiset partition is determined by the $\mathbf{j}$ and $\mathbf{j}'$.   This means that each
colored multiset partition corresponds to a matrix unit.
Hence we can define $E_{\pi^{\bc}} := E_{(\mathbf{i},\mathbf{j})}^{(\mathbf{i}',\mathbf{j}')}$ where
the pair $((\mathbf{i},\mathbf{j}), (\mathbf{i}',\mathbf{j}'))$ is the sequence of pairs
encoded by the colored multiset partition $\pi^{\bc}$.

\begin{example} If $r=5$, $k=4$ and $n=7$, let $(\mathbf{i}, \mathbf{j})
= ( (1,1),(3,1),(3,1),(5,2),(6, 4))$ and $(\mathbf{i}', \mathbf{j}')= ((5,\o1),(5,\o1),(6,\o2),(6,\o4),(2,\o4))$.
The pair $((\mathbf{i}, \mathbf{j}), (\mathbf{i}', \mathbf{j}'))$ is an index for a matrix unit
which determines a colored multiset partition
$\pi^{(1,3,5,2,6)} = \dcl \dcl 1\dcr_{\bf 1}$, $\dcl 1,1\dcr_{\bf 3}$, $\dcl 2, \o1, \o1\dcr_{\bf 5}$, $\dcl \o4 \dcr_{\bf 2}$,
$\dcl 4, \o2, \o4\dcr_{\bf 6}$ $\dcr$
of $[4]\cup [\o4]$.

The pair
$(((2,1),(4,1),(4,1),(1,2),(3, 4)), ((1,\o1),(1,\o1),(3,\o2),(3,\o4),(5,\o4)))$
is another index corresponding to the colored multiset partition is $\pi^{(2,4,1,5,3)}$
and has the same underlying multiset partition $\pi$.
\end{example}

Using the notation that we introduced at the end of Section \ref{ssec:coloredmsp}, the
product of two matrix units $E_{\pi^{\bc}}$ and $E_{\gamma^{\bc'}}$ is
\[ E_{\pi^\bc} \cdot E_{\gamma^{\bc'}}=
\delta_{\tau^{\bc''}_{top}, \pi^\bc_{top}}
\delta_{\tau^{\bc''}_{bot}, \gamma^\bc_{bot}}
\delta_{\pi^{\bc''}_{bot},\gamma^{\bc'}_{top}} E_{\tau^{\bc''}}\]
where $\delta_{A,B}$ is the Kronecker delta.  For a permutation
$\sigma \in S_n \subseteq \End(\bP^r(V_{n,k}))$
which acts on $\bP^r(V_{n,k})$, we have that $\sigma E_{\pi^\bc} = E_{\pi^{\sigma(\bc)}} \sigma$.

Let $\pi$ be a multiset partition with $r$ elements from $[k]$ and $r$ elements from $[\o{k}]$, then define
\begin{equation}\label{eq:orbitdef}
O_\pi := \sum_{\pi^{\bc} \rightarrow \pi} E_{\pi^{\bc}}
\end{equation}
where the sum runs over all  colored multiset set partitions $\pi^{\bc}$
that give rise to the same multiset partition $\pi$.  Notice that since the colors
that index the multisets are all in $[n]$, there can be at most $n$ parts in the set partition $\pi$,
otherwise the sum is empty and the element is $0$.

\begin{prop}{(The orbit basis)}  Recall that $\Pi_{r,k,n}$ is the set of all multiset partitions
with at most $n$ parts of a multiset with $r$ elements from the set $[k]$ and $r$ elements from
$[\overline{k}]$.
For $r,k,n>0$, the set of elements $\{ O_\pi \, |\,  \pi \in \Pi_{r,k,n}\}$ is a basis of $\AA_{r,k}(n)$.
\end{prop}
\begin{proof}
Since if $\pi^\bc \rightarrow \pi$, then
$\pi^{\sigma(\bc)} \rightarrow \pi$ for all $\sigma \in S_n$, it follows that
for all $\sigma \in S_n$ that $\sigma O_\pi = O_\pi \sigma$ and hence $O_\pi \in \AA_{r,k}(n)$.
By Lemma \ref{lem:elements}, we have that for any $A\in \AA_{r,k}(n)$, $A$ is a linear combination of the $O_\pi$.
Since each $O_\pi$ is the sum of a disjoint set of matrix units, these matrices are linearly independent
and therefore the collection is a basis.
\end{proof}

It follows that the dimension of $\AA_{r,k}(n)$ is equal to $|\Pi_{r,k,n}|$.
The number of multiset partitions has been previously studied and asymptotic formulas
have been obtained when the multiset partitions are counted by type,
see \cite{Ben} for details.  A generating function formula for this dimension
is given in Corollary \ref{cor:dimformula}.

\subsection{$\AA_{r,k}(n)$ is isomorphic to $\MP_{r,k}(n)$ for $n \geq 2r$}

In the previous section we showed that $\AA_{r,k}(n)$ has a basis
indexed by the elements in $\Pi_{r,k,n}$.
If $n\geq 2r$, $\Pi_{r,k,n} = \Pi_{r,k}$ and this shows that
$\AA_{r,k}(n)$ and $\MP_{r,k}(n)$ are isomorphic as vector spaces.
In this section we will show that they are isomorphic as algebras.

\begin{lemma}\label{lem:colpart}
For positive integers $r,k,n$ and $\pi, \gamma \in \Pi_{r,k,n}$,
the product on the orbit basis can be expressed as
\begin{equation} \label{eq:orbitproduct}
O_\pi \cdot O_\gamma= \sum_{\bc:\pi^\bc \rightarrow \pi} \sum_{\bc':\ga^{\bc'} \rightarrow \ga}
\sum_{\tau \in \Pi_{r,k,n}}
\delta_{\ga^{\bc'}_{top}, \pi^\bc_{bot}}
\delta_{\pi^\bc_{top}, \tau^{(1,2,\ldots,\ell(\tau))}_{top}}
\delta_{ \ga^{\bc'}_{bot}, \tau^{(1,2,\ldots,\ell(\tau))}_{bot}}
O_\tau~.
\end{equation}
\end{lemma}

\begin{proof}
Using Equation \eqref{eq:orbitdef},
we expand $O_\pi$ and $O_\gamma$ using colored multiset partitions, we have
\begin{align*}
O_\pi \cdot O_\gamma &= \sum_{\bc : \pi^\bc \rightarrow \pi}
\sum_{\bc' : \ga^{\bc'} \rightarrow \ga}
E_{\pi^{\bc}} E_{\ga^{\bc'}}\\
&= \sum_{\bc : \pi^\bc \rightarrow \pi} \sum_{\bc' : \ga^{\bc'} \rightarrow \ga}
\sum_{\tau}
\sum_{\bc'':\tau^{\bc''}\rightarrow \tau} \delta_{\ga^{\bc'}_{top},{\pi^\bc}_{bot}}
\delta_{\pi^{\bc}_{top},\tau^{\bc''}_{bot}}
\delta_{\ga^{\bc'}_{bot},\tau^{\bc''}_{bot}}
E_{\tau^{\bc''}}~.
\end{align*}
Now since the orbit basis is defined so that the support is disjoint when
expressed in the basis
elements $\{ E_{\pi^\bc} \}$ of $\End(\bP^r(M_{n,k}))$,
we only need  take a coefficient of a representative element
from the expression.  That is, in any expression in $\End_{\CC S_n}(\bP^r(M_{n,k}))$,
the coefficient of $O_\tau$ is equal to the coefficient of $E_{\tau^{\bc}}$ for
any coloring $\bc$.  Hence we can take $\bc = (1,2,\ldots,\ell(\tau))$ as a representative
element of that orbit, and Equation \eqref{eq:orbitproduct} follows.
\end{proof}

Now the following lemma shows that
the product in Equation \eqref{eq:orbitproduct} can be transformed
as a sum over elements of $\Gamma_{r,k}$, i.e., multisets with $r$ elements
from each in $[k]$, $[\o k]$ and $[\b{ \b k}]$.   Recall that a coloring $\mathbf{c}$   of
a multiset partition $\pi$ is a sequence with all distinct entries in $[n]$. The entries
have to be distinct because two blocks
of the multiset partition $\pi$ cannot have the same color.

\begin{lemma}\label{lem:bijectionexpression}
Fix positive integers $r,k,n$ and elements $\pi, \gamma, \tau \in \Pi_{r,k,n}$.
There is a bijection between pairs of colorings $(\bc, {\bc'}) \in [n]^{\ell(\pi)} \times [n]^{\ell(\gamma)}$,
each with distinct entries, of $\pi$ and $\gamma$
such that $\gamma^{\bc'}_{top} = \pi^\bc_{bot}$, $\pi^\bc_{top} = \tau^{(1,2,\ldots,\ell(\tau))}$
and $\gamma^{\bc'}_{bot} = \tau^{(1,2,\ldots,\ell(\tau))}$
and pairs $(\nu,\bc'')$ with $\nu \in \Gamma_{r,k}$, $\bc'' \in [n]^{\ell(\nu)}$ with distinct entries
such that $\nu|_{[k] \cup [\o{k}]} = \pi$,
$\nu|_{[\o{k}] \cup [\b{\b{k}}]} = \o{\gamma}$,
$\nu^{\bc''}|_{[k]\cup[\b{\b{k}}]} = \tau^{(1,2,\ldots,\ell(\tau))}$.
\end{lemma}

\begin{proof}
Consider a pair of colored multiset partitions
$(\pi^\bc, \gamma^{\bc'})$ which
satisfies $\gamma^{\bc'}_{top} = \pi^\bc_{bot}$, $\pi^\bc_{top} = \tau^{(1,2,\ldots,\ell(\tau))}$
and $\gamma^{\bc'}_{bot} = \tau^{(1,2,\ldots,\ell(\tau))}$.  This pair of colorings
uniquely determines a multiset partition $\nu \in \Gamma_{r,k}$.
That is, $\nu$ is the unique multiset partition such that there is a coloring $\bc''$
with $\nu^{\bc''}|_{[k] \cup [\o{k}]} = \pi^\bc$ and $\nu^{\bc''}|_{[\o{k}] \cup [\b{\b{k}}]} = \o{\gamma}^{\bc'}$.
To find $\nu$, start with $\pi$ and then add elements of
$[\b{\b{k}}]$ to the multisets from $\pi$ if the multiset from $\pi^\bc$
shares the same color as the multiset in $\gamma^{\bc'}$.  The multisets with only elements from $[\o{k}]$ in
$\gamma$ correspond to multisets with only elements in $[\b{\b{k}}]$ in $\nu$ and their coloring is
determined by $\gamma^{\bc'}$.
That is, the $\bc$ and $\bc'$ determine $\bc''$,
a coloring of $\nu$ such that $\tau^{(1,2,\ldots,\ell(\tau))} = \nu^{\bc''}|_{[k] \cup [\b{\b{k}}]}$.
Moreover, a coloring of $\nu$ also determines the colored multiset partitions $\pi^\bc$ and $\gamma^{\bc'}$
by reversing the process.
\end{proof}

\begin{theorem}\label{th:isomorphism}
For positive integers $n,k,r$ such that $n \geq 2r$, the algebra $\MP_{r,k}(n)$ is isomorphic to $\AA_{r,k}(n)$.
\end{theorem}

\begin{proof}
We need to show that there is a bijective homomorphism.
We know that $\{X_\pi \, |\, \pi \in \Pi_{r,k}\}$ is a basis for $\MP_{r,k}(n)$ and
$\{O_\pi \, |\, \pi \in \Pi_{r,k,n}\}$ is a basis for $\AA_{r,k}(n)$ and since $n\geq 2r$, $\Pi_{r,k} = \Pi_{r,k,n}$.
Then we define a linear map $\phi: \MP_{r,k}(n) \rightarrow \AA_{r,k}(n)$
on the basis $\phi(X_\pi) = O_\pi$ and extend this map linearly.

The product of the $O_\pi$'s is just matrix multiplication.
We will show that the coefficient of $O_\tau$ in $O_\pi O_\gamma$ is equal to
$\sum_{\nu} a_\nu b_\nu(n)$ where the sum is over all set partitions
with $\nu \in \Gamma_{r,k}$
such that $\nu|_{[k] \cup [\b{\b{k}}]} = \tau$, $\nu|_{[k] \cup [\o{k}]} = \pi$ and
$\nu|_{[\o{k}] \cup [\b{\b{k}}]} = \o{\gamma}$.
By Equation \eqref{eq:genericproduct}, this will show that $\phi$ is an isomorphism.

By Lemma \ref{lem:colpart} and the bijection in Lemma \ref{lem:bijectionexpression} we have that
the coefficient of
$O_\tau$ in $O_\pi O_\gamma$ is
\begin{align*}
\sum_{\bc:\pi^\bc \rightarrow \pi} &\sum_{\bc':\ga^{\bc'} \rightarrow \ga}
\delta_{\ga^{\bc'}_{top}, \pi^\bc_{bot}}
\delta_{\pi^\bc_{top}, \tau^{(1,2,\ldots,\ell(\tau))}_{top}}
\delta_{ \ga^{\bc'}_{bot}, \tau^{(1,2,\ldots,\ell(\tau))}_{bot}}\\
&=\sum_{\nu \in \Gamma_{r,k}} \sum_{\bc'': \nu^{\bc''} \rightarrow \nu}
\delta_{\nu|_{[k] \cup [\o{k}]},\pi}
\delta_{\nu|_{[\o{k}] \cup [\b{\b{k}}]}, \o{\gamma}}
\delta_{\nu^{\bc''}|_{[k]\cup[\b{\b{k}}]}, \tau^{(1,2,\ldots,\ell(\tau))}}~.
\end{align*}

Now fix a $\nu \in \Gamma_{r,k}$ with $\nu|_{[k] \cup [\b{\b{k}}]} = \tau$,
$\nu|_{[k] \cup [\o{k}]} = \pi$ and
$\nu|_{[\o{k}] \cup [\b{\b{k}}]} = \o{\gamma}$, and then enumerate the number of
colorings $\bc''$ such that $\nu^{\bc''}|_{[k]\cup[\b{\b{k}}]} = \tau^{(1,2,\ldots,\ell(\tau))}$.
As we will see, this expression is a polynomial in $n$.
We partition $\nu$ into two parts:  $\alpha$ which consists of the multisets which
contain at least one element from $[k] \cup [\b{\b{k}}]$, and $\beta$ which consists of those containing
only elements of $[\o{k}]$.  That is, $\nu = \alpha \uplus \beta$ and $\ell(\alpha) = \ell(\tau)$.
There are $a_\nu = \ell(\alpha)!/m(\alpha)!$ ways of coloring $\alpha$ and, once that
is completed, there are $b_\nu(n) = (n-\ell(\alpha))_{\ell(\beta)}/m(\beta)!$ ways of coloring the
parts of $\beta$ such that $\nu^{\bc''}|_{[k]\cup[\b{\b{k}}]} = \tau^{(1,2,\ldots,\ell(\tau))}$.

This shows that when $x=n \geq 2r$, the coefficient of $O_\tau$ in $O_\pi O_\gamma$
is equal to the coefficient of $X_\tau$ in $X_\pi X_\gamma$ and hence $\phi : \MP_{r,k}(n) \rightarrow \AA_{r,k}(n)$
is an algebra homomorphism.
\end{proof}

\section{Irreducible representations of $\AA_{r,k}(n)$} \label{sec:irreps}
Since  $\AA_{r,k}(n)$ is isomorphic to
the centralizer algebra of $\CC S_n$ acting
on $\bP^r( V_{n,k})$
we can use well known theorems about centralizer algebras
to begin to understand the structure of the irreducible
representations.   We begin this section with a summary of  some general results about
the representations of algebras which are centralizers of each other. Then,  we
show how these apply in the case of $\CC S_n$ and $\AA_{r,k}(n)$
acting on the space $\bP^r( V_{n,k})$.

\begin{theorem}\label{th:reptheory} \cite[Section 6.2.5]{Procesi} \cite[Section 4.2.1]{GoodWall2}
Let $A$ and $B$ be algebras acting on a module
$W$ such that $B = \End_A(W)$.
There is a set $P$ (a subset of the index set of
the irreducible representations of $A$)
such that for each $x \in P$,
$E_x$ is an irreducible $A$-module occurring in the decomposition of $W$
as a $A$-module.  Then if we set $F_x= \Hom(E_x, W)$,
then $F_x$ is an irreducible $B$-module
and the decomposition of $W$ as an $A \times B$-module is
$$W \cong \bigoplus_{x \in P} E_x \otimes F_x~.$$
Moreover, the dimension of $E_x$ is equal to
the multiplicity of $F_x$ in $W$ as a $B$-module
and the dimension of $F_x$ is equal to
the multiplicity of $E_x$ in $W$ as a $A$-module.
\end{theorem}

In particular, we are allowing $A = \CC S_n$ to act
on the polynomial ring $\bP^r( V_{n,k})$
and we know that the irreducible representations of $\CC S_n$
are indexed by the set of partitions of $n$.  Denote the irreducible
$\CC S_n$-module indexed by the partition $\lambda$ by $W^\lambda_{\CC S_n}$.
For given values of $k$ and $r$, it may be the case that  for only a subset of
the partitions $\lambda$ will $W^\lambda_{\CC S_n}$ occur with non-zero multiplicity in
$\bP^r( V_{n,k})$.  Let
\begin{equation} \label{eq:irreddef}
W^\lambda_{\AA_{r,k}(n)} :=
\Hom(W^\lambda_{\CC S_n}, \bP^r( V_{n,k}))
\end{equation}
represent the irreducible $\AA_{r,k}(n)$-module indexed by the partition $\lambda$.

For any partition $\lambda$ of $n$, we define the Reynolds operator
$R^\lambda = \frac{1}{n!} \sum_{\sigma \in S_n} \chi^{\lambda}(\sigma) \sigma$.
The operator $R^\lambda$ projects a $\CC S_n$-module
to the component of the module consisting of the direct
sum of $\CC S_n$ irreducibles indexed by the partition $\lambda$.
This means that for a $\CC S_n$-module $W$, we have
\begin{equation}\label{eq:reynoldsdecomp}
W \cong \bigoplus_{\lambda \in \PP_n} R^\lambda W~.
\end{equation}

In the case that $W = \bP^r( V_{n,k})$
the action of $\AA_{r,k}(n)$ commutes with the $S_n$ action and
$R^\lambda W$ is both a $\CC S_n$ and a $\AA_{r,k}(n)$-module and so
we have that
\begin{equation}\label{eq:reynoldstensor}
R^\lambda \bP^r( V_{n,k}) \cong
W^\lambda_{\AA_{r,k}(n)} \otimes W^\lambda_{\CC S_n}~.
\end{equation}

The following proposition gives the dimension of an irreducible $\AA_{r,k}(n)$
representation in terms of a Schur
function evaluated at a series. It also gives precise conditions for when this
dimension is non-zero.  The plethystic notation is defined in
Section \ref{subsec:sfnotation}.

\begin{prop} \label{prop:dimirred} For positive integers $n,k,r$ and $\lambda \in \PP_n$,
\begin{align}\label{eq:gfdim}
\dim W^\lambda_{\AA_{r,k}(n)} = \hbox{coefficient of }q^r\hbox{ in }s_\lambda\!\left[ \frac{1}{(1-q)^k} \right]~.
\end{align}
Moreover, this coefficient is non-zero (that is, the module $W^\lambda_{\AA_{r,k}(n)}$ exists) if and only if
$r \geq \sum_{j\geq1} (j-1) \cdot \sum_{i=t_{j-1}+1}^{t_{j}} \lambda_i$
where $t_d = \binom{k+0-1}{0} + \binom{k+1-1}{1} + \cdots + \binom{k+d-2}{d-1}$
and using the convention that $\lambda_d = 0$ for $d>\ell(\lambda)$.
\end{prop}

\begin{proof} 
Considering $\bP^r( V_{n,k}) = \bP^r( V_{n,\{1,2,\ldots,k\}})$
as a $GL_n$-module,
\begin{equation*}
\bP^r( V_{n,\{1,2,\ldots,k\}}) = \bigoplus_{r_1 + r_2 + \cdots r_k=r}
\bP^{r_1}( V_{n,\{1\}}) \otimes \bP^{r_2}( V_{n,\{2\}}) \otimes \cdots \otimes \bP^{r_k}( V_{n,\{k\}})
\end{equation*}
where the sum is over all sequences of non-negative integers $(r_1, r_2, \ldots, r_k)$
whose sum is $r$.  The $GL_n$ character of a diagonal matrix acting
on a basis $\bP^d( V_{n,\{j\}})$ is equal to $s_d[X_n]$ where
the $X_n = x_1 + x_2 + \ldots + x_n$ are a sum of $n$ variables which
can be specialized to the eigenvalues of the element of $GL_n$.
Therefore this decomposition shows that the
character of the action on $\bP^r( V_{n,\{1,2,\ldots,k\}})$ is
(by Lemma \ref{lem:sfids} part (c)),
$$\sum_{r_1 + r_2 + \cdots +r_k=r} s_{r_1}[X_n] s_{r_2}[X_n] \cdots s_{r_k}[X_n] = s_r[kX_n]~.$$

Now Scharf and Thibon's result \cite[Theorem 4.1]{ST} (that, in turn, follows from
a theorem due to Littlewood \cite[Theorem XI]{Lit}) states that the multiplicity of
an irreducible representation indexed by $\lambda$ is equal to $\left< s_r[kX], s_\lambda[\Omega[X]] \right>$~.

When evaluating symmetric functions at expressions in variables $q$,
we consider $kq = \underbrace{q+q+\cdots+q}_{k\hbox{ times}}$
then by Lemma \ref{lem:sfids} part (c), $\Omega[k q] = \Omega[q]^k = \frac{1}{(1-q)^k}$.
Since the coefficient of $q^r$ in
$\Omega[k q X] = \sum_{d \geq 0} q^d s_d[kX]$ is $s_r[kX]$,
the multiplicity that we have just calculated is equal
to the coefficient of $q^r$ in the series:
\begin{align}
\left< \Omega[k qX], s_\lambda[\Omega[X]] \right> = s_\lambda[\Omega[k q]]
= s_\lambda\!\left[\frac{1}{(1-q)^k}\right]
\end{align}
where the equalities follow by applying Lemma \ref{lem:sfids} part (b).

Fix a positive integer $k$ and a partition $\lambda$ of $n$.
To determine if the multiplicity of $W^\lambda_{\CC S_n}$ in
$\bP^r( V_{n,k})$ is non-zero
for a given $r$
(or equivalently, to determine if there is a module $W^\lambda_{\AA_{r,k}(n)}$)
we are asking if the coefficient of $q^r$ is non-zero in $s_\lambda\!\left[1/(1-q)^k\right]$.

Let $a^\lambda_{r,k}$ be equal to the number of multiset tableaux
of shape $\lambda$ with $r$ entries from $1$ to $k$.  We know by Theorem 3.1 in \cite{OZ4}
that this is the multiplicity of $W^\lambda_{\CC S_n}$.
There is an injection of the tableaux enumerated by $a^\lambda_{r,k}$ into the tableaux enumerated by $a^\lambda_{r+1,k}$ by
adding an extra entry $k$ in the highest corner of the tableau.
Hence we can say that if $a^\lambda_{r,k}$ is non-zero, then $a^\lambda_{r+1,k}$ is non-zero.

By Equation \eqref{eq:gfdim}, it remains to identify the smallest value of $r$ such that the coefficient
of $q^r$ in $s_\lambda\!\left[1/(1-q)^k\right]$ is non-zero.
First note that
\begin{align*}
1/(1-q)^k 
&=1 + \underbrace{q+q+\cdots+q}_{\binom{k+1-1}{1}~times}
+ \underbrace{q^2+q^2+\cdots+q^2}_{\binom{k+2-1}{2}~times}
+ \underbrace{q^3+q^3+\cdots+q^3}_{\binom{k+3-1}{3}~times}+ \cdots~.
\end{align*}
The expression $s_\lambda\!\left[1/(1-q)^k\right]$ is equal to
the monomial expansion of $s_\lambda[X]$ with $x_1=1$, the variables $x_2$
thought $x_{\binom{k+1-1}{1}+1}$ equal to $q$, the variables
$x_{2+\binom{k+1-1}{1}}$ through $x_{1+\binom{k+1-1}{1}+\binom{k+2-1}{2}}$
equal to $q^2$, and in general the variables $x_{t_{j-1}+1}$ through
$x_{t_j}$ equal to $q^{j-1}$.
The term of smallest degree in this expression will be the leading
term of $s_\lambda[X]$ which is $x_1^{\lambda_1} x_2^{\lambda_2} \cdots x_{\ell(\lambda)}^{\lambda_{\ell(\lambda)}}$
and that degree will be equal to
$$\sum_{j\geq1} (j-1) \cdot \sum_{i=t_{j-1}+1}^{t_{j}} \lambda_i~.$$
If $r$ is smaller than this value, then the coefficient of $q^r$ in
$s_\lambda\!\left[1/(1-q)^k\right]$ is equal to $0$, otherwise the coefficient is non-zero.
\end{proof}

\begin{remark}
The series on the right hand side of Equation \eqref{eq:gfdim}
is perhaps more explicitly written as
$$s_\lambda\!\left[ \frac{1}{(1-q)^k} \right]=
\sum_{\gamma \in \PP_n} \frac{\chi_{S_n}^\lambda(\gamma)}{z_\gamma}\prod_{i=1}^{\ell(\gamma)}\frac{1}{(1-q^{\gamma_i})^k}$$
where $\chi_{S_n}^\lambda(\gamma)$ are the values of the irreducible symmetric group characters indexed by
$\lambda$ at an element of cycle type $\gamma$.
The special case when $k=1$ is due to Aiken \cite{Aik} and in this case
the result is the multiplicity of a symmetric group irreducible in the polynomial
ring in a single set of $n$ variables (see \cite{Sta} Example 7.73 for details).

By Theorem 3.1 of \cite{OZ4},
if $n$ be a positive integer and
$\lambda \in \PP_n$, then the dimension of the irreducible
$\AA_{r,k}(n)$-module indexed by $\lambda$ (assuming that it exists)
is equal to the number of multiset valued tableaux of shape $\lambda$
with $r$ values from $1$ through $k$ (where a multiset tableaux is a
column strict tableaux whose entries are multisets ordered by a total order).
\end{remark}

\begin{cor} \label{cor:dimformula} For positive integers $n,k,r$,
the dimension of $\AA_{r,k}(n)$ is equal to the
coefficient of $z^n q^r t^r$ in
\begin{equation}\label{eq:gfdims}
\prod_{i \geq 0} \prod_{j \geq 0} \frac{1}{\left(1- z q^i t^j\right)^{\binom{k+i-1}{i} \binom{k+j-1}{j}}}~.
\end{equation}
\end{cor}

\begin{remark} \label{rem:dim}
To use this generating function for effective computation of the dimension,
one should take the finite product of $0 \leq i,j \leq r$ and
expand the series in $z$.
The coefficients of $z^n$
will be a polynomial in $q,t$ and the coefficient of $q^r t^r$
will be the dimension of $\AA_{r,k}(n)$.
However, like the partition algebra, the dimensions of the algebras grow quickly in the
parameters $k$ and $r$ and stablize for $n \geq 2r$
(e.g. as a reference point and example, $\dim \AA_{3,4}(3) = 22736$,
$\dim \AA_{3,4}(4) = 33712$, $\dim \AA_{3,4}(5) = 36912$, and
$\dim \MP_{3,4}(x) = \dim \AA_{3,4}(n) = 37312$ for $n\geq 6$).
The other coefficients of $q^a t^b$ will be the dimension of
the homomorphisms from the degree $a$ polynomials
to the degree $b$ polynomials in the variables
$x_{ij}$ with $1\leq i \leq n$ and $1 \leq j \leq k$
which commute with the action of the symmetric group.
\end{remark}

\begin{proof} (of Corollary \ref{cor:dimformula})
It is a consequence of Theorem \ref{th:reptheory},
but also by the insertion described in
\cite{COSSZ} there is a bijection between $\Pi_{r,k,n}$ and pairs of
multiset tableaux that directly show that
\begin{equation}\label{eq:dim}
\sum_{\lambda \in \PP_n} \left( \dim W^\lambda_{\AA_{r,k}(n)}\right)^2 = \dim \AA_{r,k}(n)~.
\end{equation}
We will build on this equation to prove that the coefficient of
$z^n q^r t^r$ in Equation \eqref{eq:gfdims} is equal to $\dim \AA_{r,k}(n)$.

We have by Lemma \ref{lem:sfids}, that
\begin{align*}
\prod_{i \geq 0} \prod_{j \geq 0} \frac{1}{\left(1- z q^i t^j\right)^{\binom{k+i-1}{i} \binom{k+j-1}{j}}}
&= \Omega\left[ \sum_{i,j \geq 0} \binom{k+i-1}{i} \binom{k+j-1}{j} z q^i t^j \right]\\
&= \Omega\left[ \frac{z}{(1-q)^k (1-t)^k} \right]\\
&= \sum_{n \geq 0} \sum_{\lambda \in \PP_n} z^n s_\lambda\!\left[\frac{1}{(1-q)^k}\right]
s_\lambda\!\left[\frac{1}{(1-t)^k}\right]~.
\end{align*}
Hence, by Proposition \ref{prop:dimirred} the coefficient of $z^n q^r t^r$ in the right hand side is equal
to $\sum_{\lambda \in \PP_n} \left( \dim W^\lambda_{\AA_{r,k}(n)}\right)^2 = \dim \AA_{r,k}(n)$.
\end{proof}

\subsection{The branching rule}
\label{branching}
The coefficients in the product of Schur functions
$$s_{\tau^{(1)}} s_{\tau^{(2)}} \cdots s_{\tau^{(r)}}
= \sum_{\nu} c^{\nu}_{\tau^{(1)} \tau^{(2)} \cdots \tau^{(r)}} s_\nu$$
are known as the Littlewood-Richardson coefficients.
We will use these coefficients
to express what is known as a branching rule
for the multiplicity of an irreducible for $\AA_{r-d,k-1}(n)$-module
in an irreducible of $\AA_{r,k}(n)$-module.

Note that the polynomial ring decomposes as
\begin{equation}\label{eq:polydecomp}
\bP^r( V_{n,\{1,2,\ldots,k\}})
\cong
\bigoplus_{d=0}^r
\bP^{r-d}( V_{n,\{1,2,\ldots,k-1\}})
\otimes \bP^d( V_{n,\{k\}})
\end{equation}
where the action of $\CC S_n$ is diagonal
on the tensors and we separate the variables with the second index $k$
and the parameter $d$ represents the total degree of the variables
$x_{ik}$ with $1 \leq i \leq n$.
This decomposition of the polynomial ring induces an inclusion
$\phi: \AA_{r-d,k-1}(n) \hookrightarrow \AA_{r,k}(n)$ where
for $\pi \in \Pi_{r-d,k-1,n}$, $O_\pi \in \AA_{r-d,k-1}(n)$ and
$$\phi(O_\pi) = \sum_{\nu} O_\nu$$
where the sum is over $\nu \in \Pi_{r,k,n}$
with the restriction of the entries of $\nu$ to
$[k-1] \cup [\o{k-1}]$ is $\pi$
and when $\nu$ is restricted to $\{k, \o{k}\}$
the result is a self symmetric set partition with $d$ values $k$ and $d$ values
$\o{k}$.

\begin{theorem} \label{th:branchingrule} For positive integers $r,k,n$ and
for all partitions $\lambda, \mu \in \PP_n$ and $d \geq 0$,
the multiplicity of
$W^\mu_{\AA_{r-d,k-1}(n)}$ in $W^\lambda_{\AA_{r,k}(n)}$ is equal to
$$\dim~\Hom( W^\lambda_{\CC S_n}, W^\mu_{\CC S_n}
\otimes \bP^d( V_{n,\{k\}}))
=\sum_\alpha \sum_{\tau^{(\ast)}} c^{\mu}_{\tau^{(1)} \tau^{(2)} \cdots \tau^{(d+1)}}
c^{\lambda}_{\tau^{(1)} \tau^{(2)} \cdots \tau^{(d+1)}}$$
where the sum is over sequences of non-negative integers
$\alpha = (\alpha_1, \alpha_2, \ldots, \alpha_{d+1})$
such that $\sum_{i=1}^{d+1} (i-1) \alpha_i = d$ and $\sum_{i=1}^{d+1} \alpha_i = n$
and $\tau^{(\ast)} = (\tau^{(1)}, \tau^{(2)}, \ldots, \tau^{(d+1)})$
where for each $1 \leq i \leq d+1$, $\tau^{(i)}$ is a partition of $\alpha_i$.
\end{theorem}

When $d=1$, the conditions on the sum impose that $\alpha = (n-1,1)$ and
$\sum\limits_{\tau^{(1)} \vdash n-1} c^\mu_{\tau^{(1)}(1)} c^\lambda_{\tau^{(1)}(1)}$ is
equal to the number of ways of removing a
corner cell from the partition $\mu$
and then adding a corner cell to the result to obtain the partition $\lambda$.
This is precisely the branching
rule for the partition algebra \cite{Hal,Ma4}.

\begin{remark}
Note that the analogous branching rule for $GL_k$-modules
would state how $W^\lambda_{GL_k}$ for $\lambda \in \PP_r$
decomposes as a direct sum of modules $W^\mu_{GL_{k-1}}$ where $\mu$ is
a partition of size smaller than or equal to
$r$.  In this case (see for instance \cite[Section 12.2.3]{GoodWall2}), the branching rule
is known as the Pieri rule and the multiplicity of $W^\mu_{GL_{k-1}}$ in
$W^\lambda_{GL_k}$ is $1$ if $\lambda$ and $\mu$ are `interlaced' or `differ by
a horizontal strip' and is equal to $0$ otherwise.  Theorem \ref{th:branchingrule}
is a formula for the multiplicity of
$W^\mu_{\AA_{r-d,k-1}(n)}$ in $W^\lambda_{\AA_{r,k}(n)}$
which could be greater than $1$.
\end{remark}

\begin{remark}
In \cite[Proposition 21]{OZ2} the authors gave a combinatorial interpretation
for the multiplicity appearing in Theorem \ref{th:branchingrule}
in terms of tableaux that simultaneously
satisfy two types of lattice conditions.
The application of that expression is
Theorem 13 of \cite{OZ2} which is a multiset tableaux interpretation
of
$$\dim~\Hom( W^\gamma_{\CC S_n}, W^\lambda_{\CC S_n}
\otimes \bP^{\alpha_1}( V_{n,1})\otimes \bP^{\alpha_2}( V_{n,1})\otimes \cdots \otimes \bP^{\alpha_{\ell(\alpha)}}( V_{n,1}))~.$$
\end{remark}

\begin{proof} (of Theorem \ref{th:branchingrule})
The subalgebra $\AA_{r-d,k-1}(n)$ acts exclusively on the
part of the polynomial ring where the $j$
index of the $x_{ij}$ variables has $j \neq k$, so
we have $\bP^{r-d}(V_{n,k-1})
= \bigoplus_\mu W^\mu_{\AA_{r-d,k-1}(n)} \otimes W^\mu_{\CC S_n}$.
Then by Equation \eqref{eq:irreddef} and \eqref{eq:polydecomp},
\begin{align*}
W^\lambda_{\AA_{r,k}(n)}
&\cong \bigoplus_{d=0}^r
\Hom( W^\lambda_{\CC S_n}, \bP^{r-d}( V_{n,\{1,2,\ldots,k-1\}})
\otimes \bP^d( V_{n,\{k\}}) )\\
&\cong
\bigoplus_{d=0}^r \bigoplus_{\mu \in \PP_n}
W^\mu_{\AA_{r-d,k-1}(n)} \otimes  \Hom( W^\lambda_{\CC S_n}, W^\mu_{\CC S_n}
\otimes \bP^d( V_{n,\{k\}}) )~.
\end{align*}

Therefore restricting the action of $\AA_{r,k}(n)$ to the action of
$\bigoplus_{d=0}^r \AA_{r-d,k-1}(n)$ acting
on $W^\lambda_{\AA_{r,k}(n)}$, the multiplicity of
$W^\mu_{\AA_{r-d,k-1}(n)}$ will be equal to the multiplicity
of $W^\lambda_{\CC S_n}$ in
$W^\mu_{\CC S_n}
\otimes \bP^d( V_{n,\{k\}})$.

By Proposition \ref{prop:dimirred}, the multiplicity of
$W^\nu_{\CC S_n}$ in  $\bP^d( V_{n,\{k\}})$
is the coefficient of $q^d$ in $s_\nu\!\left[\frac{1}{1-q}\right]$,
so then
$\dim~\Hom( W^\lambda_{\CC S_n}, W^\mu_{\CC S_n}
\otimes \bP^d( V_{n,\{k\}}) )$
is equal to the coefficient of $q^d$ in
$\sum_{\nu \in \PP_n} s_\nu\!\left[\frac{1}{1-q}\right] \left< s_\lambda, s_\mu \ast s_\nu \right>$,
where $\ast$ indicates the Kronecker product. Then by a change of basis calculation,
\begin{align*}
\sum_{\nu \in \PP_n} s_\nu\!\left[\frac{1}{1-q}\right] \left< s_\lambda, s_\mu \ast s_\nu \right>
&=
\sum_{\gamma\in \PP_n}
m_\gamma\!\left[\frac{1}{1-q}\right] \left< s_\lambda, s_\mu \ast h_\gamma \right>~.
\end{align*}

Since $m_\gamma\!\left[\frac{1}{1-q}\right]$ is equal to $m_\gamma[X]$ with
$x_i$ replaced with $q^{i-1}$, then
$$m_\gamma\!\left[\frac{1}{1-q}\right] = \sum_{\beta} q^{\sum_i (i-1) \beta_i}$$
where the sum is over sequences $\beta$ of non-negative integers that are permutations
of the partition $\gamma$ (with trailing zeros).
Therefore the coefficient of $q^d$ in this expression follows by applying \cite[Example 23(d), p. 130]{Mac}
to show that
$$\sum_\alpha \left< s_\lambda, s_\mu \ast h_\alpha \right>
=\sum_\alpha \sum_{\tau^{(\ast)}} c^{\mu}_{\tau^{(1)} \tau^{(2)} \cdots \tau^{(d+1)}}
c^{\lambda}_{\tau^{(1)} \tau^{(2)} \cdots \tau^{(d+1)}}$$
where the sum  is over all sequences of non-negative integers
$\alpha = (\alpha_1, \alpha_2, \ldots, \alpha_{d+1})$ that sum to $n$ such that
$\sum_{i=1}^{d+1} (i-1) \alpha_i = d$
and where the inner sum on the right hand side
is over sequence of partitions $\tau^{(\ast)} = (\tau^{(1)}, \tau^{(2)}, \ldots, \tau^{(d+1)})$
where $\tau^{(i)} \in \PP_{\alpha_i}$ for $1\leq i \leq d+1$.
\end{proof}

\subsection{Restriction of $W^\lambda_{\AA_{r,k}(n)}$}
\label{sec:restrict}
In this section we view $\AA_{d,k}(n) \otimes \AA_{r-d,\ell}(n)$ as
a subalgebra of $\AA_{r,k+\ell}(n)$. In order to define the embedding  we
need some notation suppose that $\gamma$ is a set partition of a multiset
containing only elements in
$\{k+1,k+2, \ldots, k+\ell\} \cup \{\o{k+1},\o{k+2}, \ldots, \o{k+\ell}\}$. Then
the $k$-\defn{standarization} of $\gamma$ is the multiset partition in $\Pi_{r,\ell, n}$
obtained by subtracting $k$ (resp. $\o k$) from all elements in $\gamma$.

The embedding of $\AA_{d,k}(n) \otimes \AA_{r-d,\ell}(n)$ in $\AA_{r,k+\ell}(n)$ is defined
 by mapping the element $O_\pi \otimes O_\gamma \in \AA_{d,k}(n) \otimes \AA_{r-d,\ell}(n)$,
where $\pi \in \Pi_{d,k,n}$ and $\gamma \in \Pi_{r-d, \ell, n}$, to the sum of elements $\sum_{\tau\in \Pi_{r, k+\ell, n}} O_\tau$
where the sum ranges over all $\tau$ such that the restriction of $\tau$ to $[k]\cup [\o k]$
is $\pi$ and the $k$-standardization
of $\tau$ restricted to the set $\{k+1,k+2, \ldots, k+\ell\} \cup \{\o{k+1},\o{k+2}, \ldots, \o{k+\ell}\}$ is $\gamma$.

\begin{example}
Consider the multiset partitions
$\pi = \dcl\dcl1\dcr,\dcl\o1\dcr\dcr$ and $\tau = \dcl\dcl1,\o1\dcr\dcr$,
then $O_\pi \otimes O_\tau \in \AA_{1,1}(n) \otimes \AA_{1,1}(n)$
as an element of $\AA_{2,1+1}(n)$ is equal to
$O_{\dcl\dcl1\dcr,\dcl\o1\dcr,\dcl2,\o2\dcr\dcr} +
O_{\dcl\dcl1,2,\o2\dcr,\dcl\o1\dcr\dcr} +
O_{\dcl\dcl1\dcr,\dcl2,\o1,\o2\dcr\dcr}$.
\end{example}

The next theorem shows that the multiplicity
of a tensor of irreducible multiset partition algebra
modules in the restriction from $\AA_{r,k+\ell}(n)$ to
$\AA_{d,k}(n) \otimes \AA_{r-d,\ell}(n)$ is given by
the Kronecker coefficients.  Recall that these coefficients  are the  multiplicities which
appear in the tensor of irreducible symmetric group algebra
modules.  That is, if we denote $g_{\la\nu\gamma} = \dim~\Hom( W^\lambda_{\CC S_n},
W^{\nu}_{\CC S_n} \otimes W^{\gamma}_{\CC S_n})$,
then we have that the decomposition of the restriction is
given by the following theorem.

\begin{theorem}\label{th:kronrestrict}
Let $r,k,\ell$ and $n$ be positive integers and $\lambda$ be a partition of $n$,
then the restriction of the irreducible $W^\lambda_{\AA_{r,k+\ell}(n)}$ to $\bigoplus_{d=0}^r \AA_{d,k}(n) \otimes \AA_{r-d,\ell}(n)
\subseteq \AA_{r,k+\ell}(n)$ has the decomposition
$$ {\rm Res}^{\AA_{r,k+\ell}(n)}_{\bigoplus_{d=0}^r \AA_{d,k}(n) \otimes \AA_{r-d,\ell}(n)}
W^\lambda_{\AA_{r,k+\ell}(n)} \cong
\bigoplus_{d=0}^r \bigoplus_{\nu \in \PP_n}
\bigoplus_{\gamma \in \PP_n}
(W^\nu_{\AA_{d,k}(n)} \otimes W^{\gamma}_{\AA_{r-d,\ell}(n)})^{\oplus g_{\lambda\nu\gamma}}~.
$$
\end{theorem}

\begin{proof} 
Recall from Equation \eqref{eq:irreddef} and for any $d\geq0$ that $\AA_{d,k}(n)$
is the centralizer of $\CC S_n$ acting on $\bP^d( V_{n,\{1,2,\ldots,k\}})$, so
\begin{equation}
\bP^d( V_{n,\{1,2,\ldots,k\}}) \cong \bigoplus_{\nu \in \PP_n} W^\nu_{\CC S_n} \otimes W^\nu_{\AA_{d,k}(n)}~.
\end{equation}
Therefore as an
$\CC S_n \otimes \AA_{d,k}(n) \otimes \CC S_n \otimes \AA_{r-d,\ell}(n)$-module that
\begin{align*}
\bP^d( V_{n,\{1,2,\ldots,k\}})&
\otimes \bP^{r-d}( V_{n,\{k+1,k+2,\ldots, k+\ell\}})\\
&\cong
\bigoplus_{\nu \in \PP_n} \bigoplus_{\gamma\in \PP_n}
\left(W^\nu_{\CC S_n} \otimes W^\nu_{\AA_{d,k}(n)} \right)\otimes
\left(W^\gamma_{\CC S_n} \otimes W^\gamma_{\AA_{r-d,\ell}(n)}\right)~.
\end{align*}

Now since
\begin{align*}
\bP^r( V_{n,\{1,2,\ldots,k+\ell\}})
&\cong
\bigoplus_{d=0}^r \bP^d( V_{n,\{1,2,\ldots,k\}})
\otimes \bP^{r-d}( V_{n,\{k+1,k+2,\ldots, k+\ell\}})~,
\end{align*}
then as a $\bigoplus_{d=0}^r \AA_{d,k}(n) \otimes \AA_{r-d,\ell}(n)$-module,
\begin{align*}
 &{\rm Res}^{\AA_{r,k+\ell}(n)}_{\bigoplus_{d=0}^r \AA_{d,k}(n) \otimes \AA_{r-d,\ell}(n)}
W^\lambda_{\AA_{r,k+\ell}(n)}\\
&\cong \bigoplus_{d=0}^r \bigoplus_{\nu \in \PP_n} \bigoplus_{\gamma\in \PP_n}
\Hom(W^\lambda_{\CC S_n}, W^\nu_{\CC S_n} \otimes W^\gamma_{\CC S_n})
\otimes W^\nu_{\AA_{d,k}(n)} \otimes W^\gamma_{\AA_{r-d,\ell}(n)}~.
\end{align*}
The theorem follows since $g_{\lambda\nu\gamma}$ is the dimension of
$\Hom(W^\lambda_{\CC S_n}, W^\nu_{\CC S_n} \otimes W^\gamma_{\CC S_n})$.
\end{proof}


\begin{thebibliography}{999}
\bibitem[Aik]{Aik} A. C. Aitken,
\textit{On induced permutation matrices and the symmetric group}, Proc. Edinburgh Math.
Soc. vol 5, issue 1, (1937), 1--13.

\bibitem[Bend]{Ben} E. Bender,
\textit{Partitions of multisets},
Discrete Math. 9, (1974) 301--311.

\bibitem[BH1]{BH1} G. Benkart, T. Halverson,
\textit{Partition algebras and the invariant theory of the symmetric group},
in {\it Recent trends in algebraic combinatorics}, Assoc. Women Math. Ser.,
Springer, 2019, Vol 16, pp. 1--41.

\bibitem[BH2]{BH2}
\textit{Partition algebras $\P_k(n)$ with $2k>n$ and the fundamental theorems
of invariant theory for the symmetric group $S_n$},
J. Lond. Math. Soc. (2), 99, (2019), pp. 194--224.

\bibitem[BHH]{BHH} G. Benkart, T. Halverson, N. Harmon,
\textit{Dimensions of irreducible modules for partition algebras and tensor power
multiplicities for symmetric and alternating groups},
J. Algebr. Comb. (2017) 46, pp. 77--108. {\tt https://doi.org/10.1007/s10801-017-0748-4}.

\bibitem[BDE]{BDE} C. Bowman, M. De Visscher, and J. Enyang,
\textit{The co-Pieri rule for Kronecker coefficients},
{\tt arXiv:1710.04523}.

\bibitem[BDO]{BDO} C. Bowman, M. De Visscher,  R. Orellana,
\textit{The partition algebra and the Kronecker ceofficients},
Trans. Amer. Math. Soc.  367 (2015), no. 5, 3647--3667.

\bibitem[COSSZ]{COSSZ} L. Colmenarejo, R. Orellana, F. Saliola,
A. Schilling, M. Zabrocki,
\textit{An insertion algorithm on multiset partitions
with applications to diagram algebras},
Journal of Algebra, Volume 557, 1 September 2020, Pages 97--128.
{\tt https://doi.org/10.1016/j.jalgebra.2020.04.010}

\bibitem[E]{Eny}  J. Enyang,
\textit{A seminormal form for partition algebras},
J. Combin. Theory Ser. A {\bf 120} (2013), 1737--1785.

\bibitem[F]{Frob} F. G. Frobenius,
{\it \"Uber die Charaktere der symmetrischen Gruppe},
Sitzungsberichte der K\"oniglich Preussischen Akademie Wissenschaften zu Berlin
(1900) 516--34 (Ges. Abhandlungen, {\bf 3}, 148--66).

\bibitem[Ge]{Gessel} I. Gessel,
Enumerative Applications of Symmetric Functions,
Publ. I.R.M.A. Strasbourg, 1987, 229/S-08
Actes 17e
S\'eminaire Lotharingien, p. 5--21.

\bibitem[GW1]{GoodWall1} R. Goodman, N. R. Wallach,
\textit{Representations and invariants of the classical groups},
Encyclopedia of Mathematics and its Applications,
vol. 68, Cambridge University Press, Cambridge, 1998.

\bibitem[GW2]{GoodWall2} R. Goodman, N. R. Wallach,
\textit{Symmetry, Representations, and Invariants},
Graduate Texts in Mathematics, Springer, 2009.

\bibitem[Ha]{Hal} T. Halverson,
\textit{Characters of the Partition Algebras},
J. Algebra,
Volume 238, Issue 2, 15 April 2001, pp. 502--533.

\bibitem[HJ]{HJ} T. Halverson and T.N. Jacobson,
\textit{Set-partition tableaux and representations of diagram algebras},
Alg. Comb., Volume 3 (2020) no. 2, pp. 509--538.
{\tt https://doi.org/10.5802/alco.102}

\bibitem[HL]{HL} T. Halverson, T. Lewandowski,
\textit{RSK Insertion for Set Partitions and Diagram Algebras},
Elect. J. Comb.,
Volume 11, Issue 2 (2004--6) (The Stanley Festschrift volume).

\bibitem[HR]{HR} T. Halverson, A. Ram,
\textit{Partition Algebras},
Europ. J. of Combinatorics 26 (2005) pp. 869--921.

\bibitem[Ho]{Howe}
R.~Howe, \emph{Perspectives on invariant theory: {S}chur duality,
multiplicity-free actions and beyond}, The {S}chur lectures (1992) ({T}el
{A}viv), Israel Math. Conf. Proc., vol.~8, Bar-Ilan Univ., Ramat Gan, 1995,
pp.~1--182.

\bibitem[J]{Jones} V.F.R. Jones,
\textit{The Potts model and the symmetric group},
Subfactors: Proceedings of the
Taniguchi Symposium on Operator Algebras, Kyuzeso, 1993.
World Scientific, River Edge (1994) pp. 259--267.

\bibitem[Li]{Lit} D. E. Littlewood,
\textit{Products and Plethysms of Characters with Orthogonal,
Symplectic and Symmetric Groups},
Canad. J. Math., {\bf 10}, 1958, 17--32.

\bibitem[LR]{LR} N. A. Loehr, J. B. Remmel,
\textit{A computational and combinatorial expos\'e of plethystic calculus},
J. Alg. Comb., March 2011, Volume 33, Issue 2, pp. 163--198.

\bibitem[Mac]{Mac} I.~G.~Macdonald,
\newblock \textit{Symmetric Functions and Hall Polynomials},
\newblock Second Edition, Oxford University Press,
second edition, 1995.

\bibitem[Mar1]{Ma1}
P. Martin,
\textit{Representations of graph Temperley-Lieb algebras},
Publ. Res. Inst. Math. Sci., 26 (1990), pp. 485--503.

\bibitem[Mar2]{Ma2}
P. Martin,
\textit{Potts Models and Related Problems in Statistical
Mechanics}, World Scientific, Singapore (1991).

\bibitem[Mar3]{Ma3}
P. Martin,
\textit{Temperley-Lieb algebras for nonplanar statistical
mechanics--The partition algebra construction},
J. Knot Theory Ramifications, 3 (1994), pp. 51--82.

\bibitem[Mar4]{Ma4}
P. Martin,
\textit{The structure of the partition algebras},
J. Algebra, 183 (1996), pp. 319--358.

\bibitem[ME]{ME}
P. Martin, A. Elgamal,
\textit{The structure of the partition algebras},
J. Algebra {\bf 183} (1996), 319--358.

\bibitem[MR]{MR}
P. Martin and G. Rollet,
\textit{The Potts model representation and a Robinson-Schensted
correspondence for the partition algebra},
Compositio Math. 112 (1998), 237--254.

\bibitem[NPS]{NPS} S. Narayanan, D. Paul, S. Srivastava,
\textit{The Multiset Partition Algebra}, {\tt arXiv:1903.10809}.

\bibitem[OZ]{OZ} R. Orellana, M. Zabrocki,
\textit{Symmetric group characters as symmetric functions},
{\tt arXiv:1605.06672}.

\bibitem[OZ2]{OZ2} R. Orellana, M. Zabrocki,
\textit{Products of characters of the symmetric group},
Journal of Combinatorial Theory, Series A, Volume 165,
July 2019, Pages 299--324.

\bibitem[OZ3]{OZ3} R. Orellana, M. Zabrocki,
\textit{The Hopf structure of symmetric group characters
as symmetric functions}, {\tt arXiv:1901.00378}.

\bibitem[OZ4]{OZ4} R. Orellana, M. Zabrocki,
\textit{A combinatorial model for the decomposition
of multivariate polynomial rings as $S_n$-modules},
Elect. J. Combin., to appear,
{\tt arXiv:1906.01125}.

\bibitem[OEIS]{oeis} N.~J.~A.~Sloane, \emph{The
On-Line Encyclopedia of Integer Sequences}. Published electronically
at {\tt http://oeis.org}, 2013.

\bibitem[Pa]{Paul} D. Paul,
\textit{The Multiset Partition Algebra},
Ph.D. Thesis, The Institute of Mathematical Sciences, Chennai, 2020.

\bibitem[Pr]{Procesi} C. Procesi,
\textit{Lie Groups: An approach through invariants and representations},
Springer, 2007.

\bibitem[sage]{sage} W.\thinspace{}A. Stein et~al.
\newblock \textit{{S}age {M}athematics {S}oftware
({V}ersion 6.10)},
The Sage Development Team, 2016, {\tt http://www.sagemath.org}.

\bibitem[sage-combinat]{sage-co}
The {S}age-{C}ombinat community.
\newblock  \textit{{{S}age-{C}ombinat}:
enhancing Sage as a toolbox for
computer exploration in algebraic combinatorics},
{{\tt http://combinat.sagemath.org}}, 2008.

\bibitem[ST]{ST} T. Scharf, J. Y. Thibon,
\textit{A Hopf-algebra approach to inner plethysm}.
Adv. in Math. 104 (1994), pp. 30--58.

\bibitem[Sta]{Sta}  R. Stanley,
\textit{Enumerative Combinatorics, Vol. 2}, Cambridge University Press, 1999.
\end{thebibliography}
\end{document}